\documentclass{amsart}
 
\usepackage{amsmath}
\usepackage{amsthm,amssymb,color,comment}
\usepackage{bbm}
\usepackage{hyperref}
\usepackage[TS1,T1]{fontenc}
\usepackage[utf8]{inputenc}
\usepackage{dsfont}
\usepackage{tikz}
\usepackage{enumitem}
\usepackage[sort]{cite}

\numberwithin{equation}{section}

\newtheorem{thm}{Theorem}[section]

\newtheorem{lem}[thm]{Lemma}

\newtheorem{theorem}{Theorem}[section]
\newtheorem{proposition}[theorem]{Proposition}
\newtheorem{lemma}[theorem]{Lemma}
\newtheorem{corollary}[theorem]{Corollary}
\newtheorem{definition}[theorem]{Definition}

\newtheorem{assumption}[theorem]{Assumption}
\newtheorem{remark}[theorem]{Remark}

\theoremstyle{definition}

\newtheorem{example}[theorem]{Example}

\newenvironment{proofof}[1]
{\smallskip\noindent{\textbf{Proof~of~#1.}}
\hspace{1pt}}{\hspace{-5pt}{\nobreak\nobreak\hfill\nobreak
$\square$\vspace{2pt}\par}\smallskip\goodbreak}
\newcommand{\Lip}{\mathbf{Lip}}

\newcommand{\R}{\mathbb{R}}
\newcommand{\N}{\mathbb{N}}

\newcommand{\supp}{\,\mathrm{supp}\,}

\newcommand{\diff}{\mathop{}\!\mathrm{d}}

\makeatletter
\newcommand{\doublewidetilde}[1]{{%
  \mathpalette\double@widetilde{#1}%
}}
\newcommand{\double@widetilde}[2]{%
  \sbox\z@{$\m@th#1\widetilde{#2}$}%
  \ht\z@=.9\ht\z@
  \widetilde{\box\z@}%
}
\makeatother
\author{Christian Düll}
\address{{\it Christian Düll:} Institute of Applied Mathematics, Heidelberg  University, 69120 Heidelberg, Germany }
\email{duell@math.uni-heidelberg.de}

\author{Piotr Gwiazda}
\address{{\it Piotr Gwiazda:} Institute of Mathematics of Polish Academy of Sciences, Jana i J\k edrzeja \'Sniadeckich 8, 00-656 Warsaw, Poland}
\email{pgwiazda@mimuw.edu.pl}
\thanks{Piotr Gwiazda was supported by National Science Center, Poland through project no. 2018/31/B/ST1/02289.}

\author{Anna Marciniak-Czochra}
\address{{\it Anna Marciniak-Czochra:} Institute of Applied Mathematics, Heidelberg  University, 69120 Heidelberg, Germany}
\email{anna.marciniak@iwr.uni-heidelberg.de}
\thanks{The research of AMC was supported under the Germany's Excellence Strategy EXC-2181/1 - 390900948 (the Heidelberg STRUCTURES Excellence Cluster)}

\author{Jakub Skrzeczkowski}
\address{{\it Jakub Skrzeczkowski:} Faculty of Mathematics, Informatics and Mechanics, University of Warsaw, Stefana Banacha 2, 02-097 Warsaw, Poland}
\email{jakub.skrzeczkowski@student.uw.edu.pl}
\thanks{Jakub Skrzeczkowski was supported by National Science Center, Poland through project no. 2019/35/N/ST1/03459.}

\begin{document}

\title[Measure differential equation with growth/decay]{Measure differential equation with a nonlinear growth/decay term}

\begin{abstract}
We obtain an existence result for a measure differential equation with a nonlinear growth/decay term that may change the sign. The proof requires a modification of the approximating schemes proposed by Piccoli and Rossi. The new scheme combines model discretization with an exponential solution of the nonlinear growth/decay, and hence, preserves nonnegativity of the measure. Furthermore, we formulate a new analytic condition on the measure vector field, which substantially simplifies the previous proof of continuity of solutions with respect to initial data  and generalizes the former condition formulated by Piccoli and Rossi.
\end{abstract}
\maketitle
\tableofcontents

\section{Introduction}

\noindent We consider a measure differential equation (MDE) of the form
\begin{align}
\label{eq:MDE_with_source_c}
\dot{\mu}_t = V[\mu_t] \oplus c(\cdot,\mu_t)\, \mu_t  \oplus s[\mu_t],
\end{align}
where $V: \mathcal{M}^+(\R^d) \to \mathcal{M}^+(\R^d \times \R^d)$ is the so-called measure vector field (MVF), $s: \mathcal{M}^+(\R^d) \to \mathcal{M}^+(\R^d)$ is the source term, $c: \R^d \to \R$ is a nonlinear source/decay function, and $\mathcal{M}^+(\R^d)$, $\mathcal{M}^+(\R^d \times \R^d)$ are spaces of finite, nonnegative Radon measures on $\R^d$ and $\R^d \times \R^d$, respectively. The notion using $\oplus$ is applied to depict a summation of various effects, in this case transport, source and growth/decay processes. The rigorous meaning of \eqref{eq:MDE_with_source_c} is given in Definition \ref{def:measure_solution}.  \\

\noindent Equation \eqref{eq:MDE_with_source_c} is an extension of the MDE setting for a transport of measures that was recently introduced by Piccoli in terms of a conservative transport equation \cite{MR3961299},
$$
\dot{\mu}_t = V[\mu_t],
$$
and extended by Piccoli and Rossi to an equation with a nonnegative source term \cite{MR4026977}
$$
\dot{\mu}_t = V[\mu_t] \oplus s[\mu_t].
$$

\noindent The main contribution of this paper is an existence result for \eqref{eq:MDE_with_source_c} which includes a distribution dependent growth/ decay term $c(\cdot,\mu_t)$, without any assumption on the sign of function $c$. The new MDE is motivated by applications in life sciences accounting for birth and death of individuals or cell divisions and transitions, which cannot be described solely by source terms valued in $\mathcal{M}^+(\R^d)$. Analysis of the resulting model requires modification of the original approach that was based on an approximating scheme using a proof of nonnegativity of solutions. We propose a new discretization that preserves nonnegativity of measures also in case of negative $c$. Furthermore, continuity of the resulting Lipschitz semigroup with respect to initial data and uniqueness in an appropriate class of solutions require an additional continuity condition on the measure vector field $V$ (MVF continuity condition). Exploring convergence of the approximating scheme, we formulate a new MVF continuity condition, see \eqref{eq:new_ass_BL}. The proposed setting significantly simplifies the original reasoning that was based on optimal transport, \cite{MR3961299,MR4026977}. Moreover, in Appendix \ref{appendix:piccoli proof continuity} we prove that the new MVF continuity condition generalizes the one exploited in \cite{MR3961299,MR4026977}.  \\

\noindent The proposed model \eqref{eq:MDE_with_source_c} extends the MDE framework from applications to pedestrian flows to population dynamics.  Analysis of transport and growth phenomena in the context of population dynamics has been originally performed using so called structured population models in $L^1$-setting \cite{Webb, Th2003,MeDi1986} and has then been extended to the space of measures \cite{carrillo2012,MR2997595,MR2644146,MR2746205,MR3342408,MR3507552,MR3461738}. The established theory has allowed rigorous convergence analysis of numerical algorithms such as particle methods and EBT \cite{GwJaMaUl2013, MR3986559,deRoos1988,MR3138105,carrillo2014, gwiazda2021}. Recently, it has also been  applied to show stability of {\it a posteriori} distributions obtained in Bayesian Inverse Problems \cite{szymanska2021}. The space of measures is a convenient setting for the analysis of  transport phenomena on complex domains such as graphs \cite{MR4028475,MR3657111,MR3779635} or manifolds \cite{MR3714980,rossi2016control}. Finally, there are some promising results concerning sensitivity analysis and optimal control problems that may be further combined with particle methods \cite{MR4045015, MR4066016, MR4027078}. Admitting a measure-valued velocity field, the proposed MDE model \eqref{eq:MDE_with_source_c} is an extension of the structured population models in Radon measures and may provide a new tool for model-based analysis of the evolution of heterogeneous cell populations.\\

\section{Problem formulation and main results}
 \noindent MDEs provide a generalization of the  concept of ordinary differential equations to spaces of measures. In this approach, evolution of measure $\mu$ is governed by a measure vector field $V[\mu]$ which is a measure on $\R^d \times \R^d$. Its first coordinate represents spatial position $x$ and the second denotes admissible values of velocity $v$. The measure $V[\mu]$ has marginal $\mu$ on the first coordinate, i.e. it satisfies $\pi_1^{\#} V[\mu] = \mu$ for all $\mu \in \mathcal{M}^+(\R^d)$. Here $\pi_1$ denotes the projection to the first (spatial) coordinate and the superscript $\#$ denotes the push-forward operator
    \begin{align*}
        \pi_{1}^{\#}\mu(A)=\mu\left(\pi^{-1}_1(A)\right) \qquad \forall A\in\mathcal{B}(\mathbb{R}^d).
    \end{align*}
\noindent
Roughly speaking, if $(x,v)$ belongs to the support of $V[\mu]$,  $\mu$ at position $x$ evolves with velocity $v$. We refer to Section \ref{section:continuity} and \cite[Section 7.1]{MR3961299} for examples of measure vector fields and to \cite{MR4206990} for recent work on numerical schemes for MDE. \\

\noindent The measure solution to the MDE \eqref{eq:MDE_with_source_c} is based on the weak formulation of the problem.

\begin{definition}\label{def:measure_solution}
We say that a continuous curve $\mu_{\bullet}: [0,T] \to (\mathcal{M}^+(\R^d), \| \cdot \|_{BL^*})$ is a \textbf{solution to \eqref{eq:MDE_with_source_c} with initial condition $\mu_0 \in \mathcal{M}^+(\R^d)$} if for all $f \in C_c^{\infty}(\R^d)$ and for all $t \in [0,T]$, it holds
\begin{align*}
		&\int_{\mathbb{R}^d} f(x) \diff \mu_t(x) - \int_{\R^d} f(x)\diff \mu_0(x) \\
		& \hspace{0.25cm} = \int_0^t \int_{\R^d\times \R^d} \nabla f(x) \cdot v \diff V[\mu_r](x,v) \diff r +
		\int_0^t \int_{\mathbb{R}^d} f(x) \, c(x,\mu_r) \diff \mu_r(x) \diff r	+\int_0^t\int_{\mathbb{R}^d}f(x)\diff s[\mu_r](x)\diff r.
	\end{align*}
The bullet in the subscript of the solution above denotes the time argument.
\end{definition}

\begin{remark}
In most cases the dynamics of MDEs simplifies to a transport equation in the spaces of measures as in \cite{our_book_ACPJ}. To see this, we note that by disintegration theorem, see e.g. \cite[Theorem 1.45]{MR3409135}, there exists a family of probability measures $\{ \nu_{x,t} \}_{x \in \R^d, t \in [0,T]}$ such that
$$
\int_0^t \int_{\R^d \times \R^d} \nabla f(x) \cdot v \diff V[\mu_r](x,v) \diff r = 
\int_0^t \int_{\R^d} \nabla f(x) \cdot \left[\int_{\R^d} v \diff \nu_{x,r}(v)\right] \diff x \diff r
$$
so that Definition \ref{def:measure_solution} boils down to measure solutions for transport equation with velocity 
$$
\mathcal{V}(t,x) := \left[\int_{\R^d} v \diff \nu_{x,t}(v)\right].
$$
Of course, this point of view does not bring new insights as $\left[\int_{\R^d} v \diff \nu_{x,t}(v)\right]$ cannot be computed explicitly in general. 
\end{remark}

\noindent Now, we formulate assumptions on model functions. They are expressed in $\|\cdot\|_{BL}$ and $\| \cdot \|_{BL^*}$ norm, see \eqref{eq:def_BLnorm} and \eqref{eq:BLstarnorm} respectively.

\begin{assumption}\label{ass:MVF}
\begin{itemize}
    \item [\textbf{(V)}] The measure vector field $V: \mathcal{M}^+(\R^d) \to \mathcal{M}^+(\R^d \times \R^d)$ satisfies $\pi_1^{\#} V[\mu] = \mu$ for all $\mu \in \mathcal{M}^+(\R^d)$. Moreover, it holds:
\begin{itemize}
\item\textbf{Control of support in velocity}: There is a constant $C_S > 0$ such that 
    \begin{align}
    \label{ass:support of velocity}
    \tag{$V_1$}
    \sup_{(x,v) \in \supp(V[\mu])}|v| \leq C_S\left(1+ \sup_{(x,v) \in \supp(V[\mu])}|x|\right).
    \end{align}
\item\textbf{Lipschitz continuity with respect to the flat metric}: For all $R>0$, there is a constant $C_F(R)$ such that if $\mu$, $\nu$ are supported in $B(0,R)$, we have
    \begin{align}
    \label{ass:Lipschitz cont V}
    \tag{$V_2$}
    \| V[\mu] - V[\nu] \|_{BL^*} \leq C_F(R)\,\| \mu - \nu \|_{BL^*}.
    \end{align}
\end{itemize}
\item[\textbf{(S)}] The source $s:\mathcal{M}^+(\R^d)\to\mathcal{M}^+(\R^d)$ satisfies
\begin{itemize}
\item \textbf{Lipschitz continuity}: There exists $L$ such that for all $\mu, \nu\in \mathcal{M}^+(\R^d)$
	\begin{align}
	\label{ass:Lipschitz cont s}
	\tag{$S_1$}
	\|s[\mu]-s[\nu]\|_{BL^*}\leq L\|\mu-\nu\|_{BL^*}.
	\end{align}
\item \textbf{Uniform boundedness of the support}: There exists $R$ such that for all $\mu\in \mathcal{M}^+(\R^d)$ it holds 
    \begin{align}
        \label{ass:support of s}
        \tag{$S_2$}
        \mathrm{supp}(s[\mu])\subseteq B(0,R).
    \end{align}
\end{itemize}
\item [\textbf{(C)}] The source/decay function $c:\mathbb{R}^d\times\mathcal{M}^+(\mathbb{R}^d)\to\mathbb{R}$ satisfies
\begin{itemize}
    \item \textbf{Boundedness}: There exists a constant $C_b>0$ such that 
        \begin{align}
            \label{ass:boundedness of c}
            \tag{$C_1$}
            |c(x,\mu)|\leq C_b \qquad \forall x\in \mathbb{R}^d,\mu\in \mathcal{M}^+(\R^d).
       \end{align}
    \item \textbf{Lipschitz continuity}: There exists $C_L>0$ such that
        \begin{align}
        \label{ass:Lipschitz cont c}
            \tag{$C_2$}
            |c(x,\mu)-c(y,\nu)|\leq C_L[|x-y\|+\|\mu-\nu\|_{BL^*}]. 
        \end{align}
\end{itemize}
\end{itemize}
\end{assumption}

\begin{remark}
\label{rem:local implies global}
In view of the Lemmas \ref{lemma:boundedness of support} and Corollary \ref{cor: compact support of V mu N} the constructed solution $\mu_t^N$ and the MVF $V[\mu_t^N]$ are both uniformly compactly supported so that the radius $R$ in assumption \eqref{ass:Lipschitz cont V} can be chosen uniformly. Therefore, in what follows we will write $C_F$ instead $C_F(R)$.
\end{remark}

\noindent The first main result of this paper reads: 
\begin{theorem}[Existence of solutions]
\label{thm:existence_with_c}
Suppose that the measure vector field $V$ and model functions $s, c$ satisfy Assumption \ref{ass:MVF}. Moreover, assume that $\mu_0 \in \mathcal{M}^+(\R^d)$ is compactly supported. Then, there exists a measure solution $\mu_{\bullet}$ to \eqref{eq:MDE_with_source_c} in the sense of Definition \ref{def:measure_solution}. This solution is Lipschitz continuous with respect to time.
\end{theorem}
\noindent To prove Theorem \ref{thm:existence_with_c}, we consider an approximating sequence and establish uniform bounds that enable application of the compactness argument. The main novelty of the presented analysis is the construction of approximating scheme \eqref{eq:approx_scheme_extended with s} that preserves nonnegativity of the measure even if $c$ is negative. The proof is presented in Section \ref{sec:proof_existence}.\\

\noindent Unfortunately, a uniqueness result of solutions to problem \eqref{eq:MDE_with_source_c} seems to be out of reach with Assumption \ref{ass:MVF} alone. Thus, we introduce an additional assumption on the measure vector field $V$.
\begin{assumption}
[MVF continuity condition]\label{ass:new_cont_condition} For all $R>0$, there is a constant $C_H(R)$ such that if $\mu$, $\nu$ are supported in $B(0,R)$, we have
\begin{align}\label{eq:new_ass_BL}
\tag{$V_3$}
\sup_{\|\psi\|_{BL\left(\R^d\right)} \leq 1}\int_{\R^d \times \R^d} \psi(x + \tau \, v) \diff(V[\mu] - V[\nu])(x,v) \leq (1 + C_H(R)\,\tau) \, \| \mu - \nu \|_{BL^*}.
\end{align}
\end{assumption}

\begin{remark}
\label{rem:discussion assumption V2hat for existence}
Note that assumption \eqref{ass:Lipschitz cont V}, which is necessary for the existence result Theorem \ref{thm:existence_with_c}, only implies \eqref{eq:new_ass_BL} if $C_F\leq 1$. So in general, both \eqref{ass:Lipschitz cont V} and \eqref{eq:new_ass_BL} have to be assumed to establish existence and uniqueness of solutions to model \eqref{eq:MDE_with_source_c}. Furthermore, the same reasoning as in Remark \ref{rem:local implies global} applies, so that we can drop the radius $R$ appearing in the constant $C_H(R)$ in \eqref{eq:new_ass_BL} and simply write $C_H$ instead.
\end{remark}
\noindent Under this additional MVF continuity assumption, the corresponding semigroup of solutions proves to be continuous with respect to initial conditions. 
We remark that the alternative Lipschitz continuity condition with respect to the operator $\mathcal{W}^g$ (see Definition \ref{def:operator Wg}) was  formulated by Piccoli in \cite{MR4026977} and applied to obtain continuity with respect to initial conditions in \cite{MR3961299,MR4026977}. We show that this approach is a special case of our reasoning in Appendix \ref{appendix:piccoli proof continuity}.

\begin{theorem}
\label{theorem: dependence initial data with new assumption}
Suppose that Assumptions \ref{ass:MVF} and \ref{ass:new_cont_condition} are satisfied. Let $\mu_0,\nu_0\in \mathcal{M}^+_c(\R^d)$ be two initial data with corresponding solutions $\mu_t,\nu_t$ to problem \eqref{eq:MDE_with_source_c}. Then it holds that
    \begin{align}
      \label{eq:dependence initial conditions}
        \|\mu_t-\nu_t\|_{BL^*}\leq e^{Ct}\|\mu_0-\nu_0\|_{BL^*}.
    \end{align}
\end{theorem}

\noindent Uniqueness of the measure solution $\mu_{\bullet}$ in appropriate class can then be established as in \cite{MR3961299,MR4026977} and is formulated in Theorem \ref{thm:uniqueness}.\\

\noindent The structure of the paper is as follows.
In Section \ref{section:flat_norm} we introduce the flat norm and  some results on compactness of measures, which will be used in the proof of Theorem \ref{thm:existence_with_c}. Section \ref{section:construction of solution} is devoted to the explicit construction of the lattice approximate solution $\mu_t^N$. Furthermore, we establish useful bounds and estimates for the supports. In Section \ref{sec:proof_existence} we prove the existence result (Theorem \ref{thm:existence_with_c}). We continue in Section \ref{section:continuity} with our second main result on continuity of solutions with respect to initial data (Theorem \ref{theorem: dependence initial data with new assumption}). For the proof we need to upgrade assumption \eqref{ass:Lipschitz cont V} with the additional regularity hypothesis \eqref{eq:new_ass_BL}.  In Section \ref{sect:lip_semigroup_uniqueness} we summarise the  theory introduced in \cite{MR3961299,MR4026977} to show uniqueness of the resulting semigroup based on the concept of Dirac germs (see Definition \ref{def:dirac germ and compatible} and Theorem \ref{thm:uniqueness}). Additionally, in Appendix \ref{appendix:piccoli proof continuity},  we prove that our MVF continuity condition $(V_3)$ generalizes the one exploited formerly in \cite{MR3961299,MR4026977}.

\section{Flat norm on the space of measures}
\label{section:flat_norm}
\noindent In this section, we present our functional analytic setting. Let $\mathcal{M}(\R^d)$ be the space of bounded real-valued signed Borel measures on $\R^d$ and let $\mathcal{M}^+(\R^d)$ be the cone consisting of nonnegative measures cf. \cite[Sections 1.3, 3.1]{Folland.1984}. The space of compactly supported nonnegative measures is denoted by $\mathcal{M}^+_c(\R^d)$. We can define a partial ordering on $\mathcal{M}^+(\R^d)$ via
    \begin{align*}
        \mu\leq\nu:\Leftrightarrow \mu(A)\leq \nu(A)\qquad \text{for all } A\in \mathcal{B}(\R^d).
    \end{align*}
We recall Hahn-Jordan decomposition: If $\mu \in \mathcal{M}(\R^d)$ is a signed measure, there are  two (uniquely determined) nonnegative measures $\mu^+, \mu^- \in \mathcal{M}^+(\R^d)$ with disjoint supports such that
$$
\mu = \mu^+ - \mu^-.
$$ 
To perform analysis on $\mathcal{M}(\R^d)$, we need a notion of distance. The standard one is given by the total variation norm:
\begin{align*}\label{eq:TVnorm}
		\|\mu\|_{TV}:= \mu^+(\R^d) + \mu^{-}(\R^d).
\end{align*}

\noindent Unfortunately, total variation generates a topology which is too strong for applications \cite[Examples 1.1,1.2]{MR2746205}, so in this paper we will work in spaces equipped with the \textbf{flat norm} (or \textbf{bounded Lipschitz distance}, \textbf{Fortet-Mourier distance}) defined as
\begin{equation}\label{eq:BLstarnorm}
		\|\mu\|_{BL^*}:= \sup\left\{\int_{\R^d} \! \psi  \,\mathrm{d}\mu \mid \psi \in BL(\R^d), \|\psi\|_{BL} \leq 1\right\}.
\end{equation} 
The \textbf{space of bounded Lipschitz functions} $BL(\R^d)$ is given by
\begin{equation*}\label{eq:def_BL}
	BL(\R^d)=\left\{f:\R^d \to \R \mbox{ is continuous and } \|f\|_{\infty}<\infty, |f|_{\Lip}<\infty\right\},
\end{equation*}
 where  
 \begin{equation*}\label{app:basic_norms}
 \|f\|_{\infty}=\underset{x\in \R^d}{\sup}\,|f(x)|, \qquad \qquad |f|_{\Lip}=\underset { x\neq y}{\sup}\, \frac {|f(x)-f(y)|}{|x-y|}.
 \end{equation*}
Equipped with the norm   
\begin{equation}\label{eq:def_BLnorm}
\|f\|_{BL} = \max\left(\|f\|_{\infty}, \, |f|_{\Lip}\right) \leq \|f\|_{\infty} + |f|_{{\Lip}},
\end{equation}
the space $(\mathcal{M}^+(\R^d), \| \cdot \|_{BL^*})$ is a separable and complete metric space \cite[Corollary 1.38, Theorem 1.61]{our_book_ACPJ} or \cite[Theorem 2.7 (ii)]{MR2644146}. We also remark that if $\mu \in \mathcal{M}^+(\R^d)$, then $\|\mu\|_{TV} = \| \mu \|_{BL^*}$.
\begin{remark}
\label{rem:Piccoli characterisation of flat metric}
In \cite[Theorem 13]{piccoli2014generalized}, the following alternative characterization for the flat norm of two measures $\mu,\nu\in \mathcal{M}^+(\R^d)$ is proven
    \begin{align*}
     \|\mu-\nu\|_{BL^*}=W_{1,1}^1(\mu,\nu):= \inf_{\substack{\tilde \mu\leq \mu,\,\tilde\nu\leq \nu\\\|\tilde\mu\|_{TV}=\|\tilde\nu\|_{TV}}}\|\mu-\tilde \mu\|_{TV}+\|\nu-\tilde\nu\|_{TV}+W_1(\tilde \mu,\tilde\nu).
    \end{align*}
Here, $W_1$ denotes the classical $1-$Wasserstein distance with respect to the cost function $c(x,y)=|x-y|$. The decomposition into terms with total variation and the term with  Wasserstein distance admits a heuristic interpretation: any share $\delta \mu$  of the mass of $\mu$ can either be transported from $\mu$ to $\nu$ at cost  $W_1(\delta \mu,\delta \nu)$ or removed at cost $\|\delta \mu\|_{TV}$. As such, the minimal "sub-measures" $\tilde \mu,\tilde \nu$ achieve an optimal compromise between mass transportation and cancellation.
\end{remark}

\noindent In this paper we will use Arzel\`a-Ascoli theorem \cite[Theorem 9.4.13]{MR3244279} in the space of measures. We briefly discuss the technical details below. 

\begin{thm}[Arzel\`a-Ascoli]
\label{thm:arzela_ascoli}
Let $X$ be a separable metric space and $Y$ be a complete metric space. Let $\mathcal{F} \subset C(X,Y)$ be a family of continuous functions such that
\begin{itemize}
    \item $\mathcal{F}$ is equibounded,
    \item $\mathcal{F}$ is equicontinuous,
    \item for each $x \in X$, the set $\{f(x)\mid f \in \mathcal{F}\}$ is relatively compact in $Y$.
\end{itemize}
Then, each sequence of functions $(f_n)_{n \in \N} \subset \mathcal{F}$ has a subsequence converging uniformly on compact subsets of $X$.
\end{thm}

\noindent In our case $X = [0,T]$ is the time interval while $Y=(\mathcal{M}^+(\R^d), \| \cdot\|_{BL^*})$ is the space of nonnegative measures equipped with the flat metric. To verify pointwise compactness, we will use the following result.

\begin{lem}
\label{lemma:pointwise_compactness}
Suppose that $(\mu_n)_{n \in \N}$ is a sequence of nonnegative measures supported on some compact set $K \subseteq \R^d$ such that $\|\mu_n\|_{TV} \leq C$ for all $n\in \mathbb{N}$. Then, there exists a subsequence $(\mu_{n_k})_{k \in \N}$ converging to $\mu$ in $\|\cdot\|_{BL^*}$ norm.
\end{lem}

\begin{proof}
First, as the sequence of measures $(\mu_n)_{n \in \N}$ is supported on the compact set $K \subseteq \R^d$, it holds that $\mu_n(\R^d\setminus K)=0$ for all $n\in \mathbb{N}$ so that $(\mu_n)_{n \in \N}$ is tight. Therefore the theorem of Prokhorov \cite[Theorem 2.3]{MR1207136} implies that there is a subsequence converging narrowly to a measure $\mu\in \mathcal{M}^+(\R^d)$, i.e. for all continuous and bounded functions $\psi: \R^d \to \R$ we have
$$
\int_{\R^d} \psi(x) \diff \mu_{n_k}(x) \to \int_{\R^d} \psi(x) \diff \mu(x).
$$
But then $\|\mu_{n_k} - \mu \|_{BL^*} \to 0$ as $k \to \infty$ according to \cite[Theorem 2.10 (ii)]{MR2746205} or \cite[Theorem 1.57]{our_book_ACPJ}.

\end{proof}

\section{Construction of the solution}
\label{section:construction of solution}
\noindent To construct a solution to \eqref{eq:MDE_with_source_c}, we will define an approximating scheme which is based on discretization of time, space and velocity. We use the notation of Piccoli for meshes and $\Delta_N$ for a mesh step. In particular, for $N\in \mathbb{N}$ the \textbf{time step size} is given by $\Delta_N= 1/ N$, the \textbf{velocity step size} by $\Delta_N^v= 1/ N$ and the \textbf{space step size} by $\Delta_N^x= 1/ {N^2}$. Furthermore, we use the same equispaced \textbf{space mesh} $(\mathbb{Z}^d/(N^2))\cap[-N,N]^d$ with discretization points $x_i$, $i=1,...,I=I(N):=(2N^3+1)^d$ and the same equispaced \textbf{velocity mesh} $(\mathbb{Z}^d/N)\cap[-N,N]^d$ with discretization points $v_j$, $j=1,...,J=J(N):=(2N^2+1)^d$. The time interval $[0,T]$ is divided into $M+1$ subintervals (with $M=M(N):=\lfloor T\rfloor N+\left\lceil(T-\lfloor T\rfloor)/ N \right\rceil-1)$  of length at most $\Delta_N$, where  the intervals are of the form $[t_l,t_{l+1})$ with $t_l=l/N$, for $l=0,...,M-1$, and the last one is given by $[t_{M},t_{M+1}]$ with $t_{M+1}=T$.\\ 

\noindent Using the above mesh, we can introduce the following discretization operators in the space and in the velocity variable
    \begin{align*}
    &\mathcal{A}^x_N:\mathcal{M}^+(\R^d)\to\mathcal{M}^+(\R^d): &&\mathcal{A}_N^x(\mu)=\sum_{i=1}^Im_i^x(\mu)\delta_{x_i},\\
    &\mathcal{A}^v_N:\mathcal{M}^+\left((\R^d)^2\right)\to\mathcal{M}^+\left((\R^d)^2\right): &&\mathcal{A}_N^v(V)=\sum_{i=1}^I\sum_{j=1}^Jm_{ij}^v(V)\delta_{(x_i,v_j)},
    \end{align*}
where we used the following measure dependent weights
  \begin{align*}
       & m_i^x(\mu)=\mu\left(x_i+[0,\Delta_N^2)^d\right):=\mu(x_i+Q):=\mu(Q_i) \qquad \text{and} \\\ &m_{ij}^v(V)=V\left((x_i+[0,\Delta_N^2)^d)\times\left(v_j+[0,\Delta_N)^d\right)\right):=V\left((x_i+Q)\times (v_j+Q')\right):=V(Q_{i,j}),
    \end{align*}
 Now, starting with an initial measure $\mu_0\in \mathcal{M}^+_c(\R^d)$, we can define the \textbf{lattice approximate solution} $\mu_t^N$: At $t=0$ we set
    \begin{align*}
    \mu_0^N=\mathcal{A}_N^x(\mu_0)
    \end{align*}
and for a time mesh point $t_l$ and  $\tau\in [0,\Delta_N]$ we set via recursion 
	\begin{align}
	\label{eq:approx_scheme_extended with s}
\mu_{t_l + \tau}^N =\tau \sum_{i=1}^Im_i^x(s[\mu_{t_l}^N])\delta_{x_i}+ \sum_{i=1}^I \sum_{j=1}^J m_{i,j}^v\left(V[\mu_{t_l}^N]\right)\, \delta_{x_i + \tau v_j}\, e^{c(x_i,\mu_{t_l}^N)\,\tau}.
	\end{align}
Note that $\mu_{t_l + \tau}^N$ is a nonnegative measure, independent of the sign of $c$.\\

\noindent We adapt \cite[Proposition 19]{MR4026977} to our setting.
\begin{proposition}
\label{prop:adaptation of Prop 19}
Let $\mu\in \mathcal{M}^+_c(\R^d)$. Then for $N$ sufficiently large it holds
	\begin{align*}
	\left\|\mu-\mathcal{A}_N^x(\mu) \right\|_{BL^*}\leq \sqrt{d}\,\Delta_N^2\|\mu\|_{BL^*}
	\end{align*}
Similarly, 
$$
	\left\|V[\mu]-\mathcal{A}_N^v(V[\mu]) \right\|_{BL^*}\leq 2\sqrt{d} \, \Delta_N\|\mu\|_{BL^*}.
$$
\end{proposition}
	
\begin{proof}
Let $\psi\in BL(\R^d)$ with $\|\psi\|_{BL}\leq 1$ and $N\in \mathbb{N}$ so large that $\mathrm{supp}(\mu)\subseteq [-N,N]^d$. Then
	\begin{multline*}
	\int_{\R^d}\psi(x)\diff\bigg(\mu-\sum_{i=1}^Im_i^x\delta_{x_i}\bigg)=\sum_{\iota=1}^I\int_{Q_{\iota}}\psi(x)\diff\bigg(\mu-\sum_{i=1}^Im_i^x\delta_{x_i}\bigg)\\
	=\sum_{i=1}^I\int_{Q_i}\psi(x)-\psi(x_i)\diff \mu
	\leq\|\psi\|_{BL}\sum_{i=1}^I\int_{Q_i}\|x-x_i\|\diff \mu
	\leq \sqrt{d}\Delta_N^2\|\mu\|_{BL^*}.
	\end{multline*}
The other statement follows similarly. Just note that 
    \begin{align*}
        \|(x,v)-(x_i,v_j)\|\leq\sqrt{d\frac 1 {N^4}+d\frac 1 {N^2}}\leq 2\sqrt{d}\Delta_N.
    \end{align*}
\end{proof}

\noindent The following lemma is a simple consequence of the push-forward condition in the Definition of the measure vector fields.

\begin{lemma}\label{lem:connection_mu_Vmu}
Let $\mu \in \mathcal{M}^+(\R^d)$. Then, $\|V[\mu]\|_{TV} = \| \mu \|_{TV}$. Moreover, if for some $\{x_i\}_{i = 1,...,I}$, $\{v_j\}_{j = 1,...,J} \subset \R^d$ we have
$$
V[\mu] = \sum_{i=1}^I\sum_{j=1}^J m_{i,j} \delta_{(x_i, v_j)},
$$
then $\mu = \sum_{i=1}^I m_i \delta_{x_i}$ and $m_i = \sum_{j=1}^J m_{i,j}$.
\end{lemma}

\begin{proof}
For all test functions $\psi\in BL(\R^d)$ we have 
$$
\int_{\R^d} \psi(x) \diff \mu(x) = \int_{\R^d \times \R^d} \psi(\pi_1(x,v)) \diff V[\mu](x,v).
$$
The first part of the lemma follows from taking $\psi = 1$. For the second part we observe that
$$
\int_{\R^d} \psi(x) \diff \mu(x) = \int_{\R^d \times \R^d} \psi(\pi_1(x,v)) \diff V[\mu](x,v) = \sum_{i=1}^I\sum_{j=1}^J m_{i,j} \psi(x_i).
$$
Considering $\psi$ which vanishes at $\{x_i\}_{i = 1,...,I}$ we obtain that $\mu$ may be supported only at points $\{x_i\}_{i = 1,...,I}$. Hence, taking $\psi$ which is one at $x_i$ and vanishes at $x_k$ for $k \neq i$ we conclude the proof.
\end{proof}




\begin{lemma}
\label{lemma:boundedness of support}
Let $\tilde R=\max\{R,R_0\}$, where $R$ is the maximal radius corresponding to the support of $s$ and $R_0$ is chosen such that $\mathrm{supp}(\mu_0)\subseteq B(0,R_0)$. Then for all $l=0,...,M$, $\tau\in [0,\Delta_N]$ and $N\in \mathbb{N}$ big enough
    \begin{align*}
        \mathrm{supp}(\mu_{t_l+\tau}^N)\subseteq B\left(0,e^{C_ST}(\tilde R+2)-1\right).
    \end{align*}
In particular, $\mu_{t_l+\tau}^N$ has a compact support $K$ which is independent of $N\in \mathbb{N}$, $l\in \{0,...,M\}$ and $\tau\in [0,\Delta_N]$.
\end{lemma}

\begin{proof}
We first note the following auxiliary statement: If $\mu\in \mathcal{M}_c^+(\R^d)$ with $\mathrm{supp}(\mu)\subseteq B(0,r)$, then 
    \begin{align}
        \label{eq:support of sum m i x mu}
        \supp\left(\mathcal{A}_N^x(\mu)\right)\subseteq B(0,r+1).
    \end{align}
Indeed, $m_i^x(\mu)=\mu(Q_i)=0$ if $Q_i\nsubseteq B(0,r)$ or equivalently if $x_i\notin B(0,r+\sqrt{d}/N^2)\subseteq B(0,r+1)$ for $N$ big enough, which was to show.\\
\noindent From \eqref{eq:support of sum m i x mu} and the definition of $\mu_0^N$ it follows directly that 
    \begin{align*}
        \supp(\mu_0^N)\subseteq\supp(0,\tilde R+1).
    \end{align*}
Now let $l\in \{1,...,M\}$ and $\tau\in[0,\Delta_N]$. Suppose $\supp(\mu_{t_l}^N)\subseteq B(0,R_l^N)$ with $R_l^N\geq \tilde R+1$. Then we claim that
    \begin{align}
        \label{eq:support of mu t l+tau}
        \supp(\mu_{t_l+\tau}^N)\subseteq B(0,R_l^N+\Delta_NC_S(1+R_l^N)).
    \end{align}
We consider the summands in \eqref{eq:approx_scheme_extended with s} separately. For the first term, we note that by \eqref{ass:support of s} $\supp(s[\mu_{t_l}^N])\subseteq B(0,\tilde R)$ and thus $\supp\left(\tau \sum_{i=1}^Im_i^x(s[\mu_{t_l}^N])\right)\subseteq B(0,\tilde R+1)\subseteq B(0,R_l^N)$ by \eqref{eq:support of sum m i x mu}. For the second term, we invoke \eqref{ass:support of velocity} and see that if $(x_i,v_j)\in \supp V[\mu_{t_l}^N]$ then the assumption implies $|v_j|\leq C_S(1+R_l^N)$ and consequently $\supp(\delta_{x_i+\tau v_j})\subseteq B(0,R_l^N+\Delta_N C_S(1+R_l^N))$. Now, \eqref{eq:support of mu t l+tau} follows directly. By induction over $l$ it can be shown that for all $k=0,...,l$
    \begin{align}
        \label{eq support estimate induction}
        R_l^N\leq R_{l-k}^N(1+\Delta_NC_S)^k+(1+\Delta_N C_S)^k-1,
    \end{align}
where the induction base follows from \eqref{eq:support of mu t l+tau} choosing $\tau=\Delta_N$.\\

\noindent We conclude the proof by applying \eqref{eq support estimate induction} with $k=l$ and using that $R_0^N=R_0\leq \tilde R+1$
    \begin{align*}
        R_l^N\leq R_0^N(1+\Delta_NC_S)^l+(1+\Delta_NC_S)^l-1
        \leq (\tilde R+1)e^{\Delta_NC_Sl}+e{\Delta_NC_Sl}-1
        \leq e^{TC_S}(\tilde R+2)-1,
    \end{align*}
since $l\Delta_N\leq T$.
\end{proof}

\begin{lemma}
\label{lem:properties_scheme_new}
Let $t \in [0,T]$ and let $\mu_t^N$ be defined by \eqref{eq:approx_scheme_extended with s}.  
If $(x, v) \in \supp(V[\mu_{t}^N])$ for some $t\in [0,T]$, then  
    \begin{align*}
    |v| \leq C_S\, e^{C_ST} \, (\tilde R + 2),
    \end{align*}
where $\tilde R$ has been defined in Lemma \ref{lemma:boundedness of support}.
\end{lemma}

\begin{proof}
Before we prove the statement, we note the following: If $\mu\in \mathcal{M}^+_c(\R^d)$, then 
    \begin{align}
    \label{support of V mu}
        \supp(V[\mu])\subseteq \supp(\mu)\times \R^d. 
    \end{align}
Indeed, if $A\nsubseteq \supp(\mu)$, $A\subseteq \R^d$, then
    \begin{align*}
        0=\mu(A)=V[\mu]\circ \pi_1^{-1}(A)=V[\mu](A\times \mathbb{R}^d)
    \end{align*}
and the claim follows by contraposition.\\

\noindent Now we prove the Lemma. According to \eqref{ass:support of velocity}, \eqref{support of V mu} and Lemma \ref{lemma:boundedness of support}
    \begin{align*}
    |v_j|\leq& \,C_S\left(1+\sup_{(x,v)\in \supp(V[\mu_t])}|x|\right)\leq C_S\left(1+\sup_{x\in \supp(\mu_t)}|x|\right)\\
    \leq &C_S\left(1+e^{C_ST}(\tilde R+2)-1\right) = \, C_S \, e^{C_ST}(\tilde R+2).
    \end{align*}
\end{proof}

\noindent Combining the results of Lemma \ref{lemma:boundedness of support} and \ref{lem:properties_scheme_new} with equation \eqref{support of V mu} we note the following:
\begin{corollary}
\label{cor: compact support of V mu N}
$V[\mu_{t_l}^N]$ is compactly supported in some set $K_V$ which is independent of $l$ and $N$.
\end{corollary}

\begin{lemma}[Lipschitz continuity of $t\mapsto \mu_t^N$ and further estimates]\label{lem:lipschitz_cont_bounds}
There is a constant $C_d$ (independent of $N$ and $t \in [0,T]$) such that
\begin{equation}\label{eq:continuity_target_proof}
\| \mu_t^N - \mu_s^N \|_{BL^*} \leq C_d\, |t-s|.
\end{equation}
Moreover, we have estimates
\begin{equation}\label{eq:uniform_bounds_TV_sum}
\|V[\mu_{t}^N]\|_{BL^*} = \|\mu_{t}^N\|_{BL^*}\leq C_d, \qquad \qquad \sum_{i=1}^I\sum_{j=1}^J  m_{i,j}^v(V[\mu^N_{t_l}]) \,(1+|v_j| + |v_j|^2) \leq C_d.
\end{equation}
\end{lemma}
\begin{proof}
First, consider $l\in \{0,...,M\}$ and $\tau_1, \tau_2 \in [0,\Delta_N]$. Let $\psi \in BL(\R^d)$ with $\|\psi\|_{BL} \leq 1$. Using the representation \eqref{eq:approx_scheme_extended with s} we obtain
\begin{multline*}
\int_{\R^d} \psi(x) \diff\!\left(\mu^N_{t_l + \tau_1} - \mu^N_{t_l +\tau_2} \right)(x) = (\tau_1-\tau_2)\sum_{i=1}^Im_i^x(s[\mu_{t_l}^N]) \, \psi(x_i) \\
+ \sum_{i=1}^I\sum_{j=1}^J  m_{i,j}^v(V[\mu^N_{t_l}]) \left[\psi(x_i + \tau_1 v_j) \, e^{c(x_i,\mu_{t_l}^N)\tau_1} - \psi(x_i + \tau_2 v_j) \, e^{c(x_i,\mu_{t_l}^N)\tau_2} \right].
\end{multline*}
Now, the second term above can be further rewritten as
\begin{multline*}
\sum_{i=1}^I\sum_{j=1}^J  m_{i,j}^v(V[\mu^N_{t_l}]) \left[\psi(x_i + \tau_1 v_j) - \psi(x_i + \tau_2 v_j) \right]  e^{c(x_i,\mu_{t_l}^N)\tau_1}  \\
+\sum_{i=1}^I\sum_{j=1}^J  m_{i,j}^v(V[\mu^N_{t_l}]) \,\psi(x_i + \tau_2 v_j)\, \left[e^{c(x_i,\mu_{t_l}^N)\tau_1} -  \, e^{c(x_i,\mu_{t_l}^N)\tau_2} \right].
\end{multline*}
Using Lipschitz continuity of $\psi$ and of the exponential function (on the bounded interval $[0,\Delta_N]$ with constant $\|c\|_{\infty}e^{\|c\|_{\infty}\Delta_N}$) we obtain
\begin{equation}\label{eq:difference estimation}
\left|\int_{\R^d} \psi(x) \diff\!\left(\mu^N_{t_l + \tau_1} - \mu^N_{t_l +\tau_2} \right)(x)\right| \leq |\tau_1-\tau_2|\left[B_1^l+B_2^l+B_3^l\right],
\end{equation}
where 
$$
B_1^l := \sum_{i=1}^Im_i^x(s[\mu_{t_l}^N])\psi(x_i), \qquad B_2^l := \sum_{i=1}^I\sum_{j=1}^J  m_{i,j}^v(V[\mu^N_{t_l}]) \, |\psi|_{\Lip} \, |v_j| \, e^{\|c\|_{\infty}\Delta_N},
$$
$$
 B_3^l = \sum_{i=1}^I\sum_{j=1}^J  m_{i,j}^v(V[\mu^N_{t_l}]) \,\| \psi \|_{\infty} \, e^{\|c\|_{\infty}\Delta_N} \, \|c\|_{\infty}.
$$
We want to estimate the terms $B_1^l,B_2^l$ and $B_3^l$ and
start with $B_2^l$ and $B_3^l$. According to Lemma \ref{lem:properties_scheme_new} $|v_ j|$ is uniformly bounded and by Lemma \ref{lem:connection_mu_Vmu}
    \begin{align*}
       \sum_{i=1}^I\sum_{j=1}^J  m_{i,j}^v(V[\mu^N_{t_l}])\leq\|V[\mu_{t_l}^N]\|_{BL^*}=\|\mu_{t_l}^N\|_{TV}=\|\mu_{t_l}^N\|_{BL^*},
    \end{align*}
so the sum can  be  controlled by $\|\mu_{t_l}^N\|_{BL^*}$. Thus, we obtain a constant $C$ (independent of $\psi$) such that for all $l \leq M$ and $\tau_1, \tau_2 \in [0,\Delta_N]$ 
    \begin{align}
    \label{bound on B2 B3}
    |B_2^l+B_3^l|\leq C\|\mu_{t_l}^N\|_{BL^*}.
    \end{align}
Next, we try to bound $B_1^l$
\begin{equation}\label{bound on B1}
\begin{split}
	&\sum_{i=1}^Im_i^x(s[\mu_{t_l}^N])\psi(x_i)\leq \|\psi\|_{BL}\left|\sum_{i=1}^Im_i^x(s[\mu_{t_l}^N])\right|
	\leq \left\|s[\mu_{t_l}^N]\right\|_{BL^*}\\
&\qquad\qquad	\leq\left[\|s[\mu_{t_l}^N]-s[\mu_0]\|_{BL^*}+\|s[\mu_0]\|_{BL^*}\right]
	\leq\left[L\|\mu_{t_l}^N-\mu_0\|_{BL^*}+\|s[\mu_0]\|_{BL^*}\right].
	\end{split}
	\end{equation}
Note that we used \eqref{ass:Lipschitz cont s} in the last inequality. Plugging \eqref{bound on B2 B3} and \eqref{bound on B1} into \eqref{eq:difference estimation} and taking the supremum over all $\psi$ leads to 
	\begin{equation}
	\label{eq:equicontinuity_one_time_step}
	\begin{split}
	\|\mu_{t_l+\tau_1}^N-\mu_{t_l+\tau_2}^N\|_{BL^*}&\leq |\tau_1-\tau_2|\left[C\|\mu_{t_l}^N\|_{BL^*}+L\|\mu_{t_l}^N-\mu_0\|_{BL^*}+\|s[\mu_0]\|_{BL^*}\right]\\
	&\leq |\tau_1-\tau_2|\left[C\|\mu_{0}\|_{BL^*}+(L+C)\|\mu_{t_l}^N-\mu_0\|_{BL^*}+\|s[\mu_0]\|_{BL^*}\right]\\
	&\leq |\tau_1-\tau_2|\left[C+C\|\mu_{t_l}^N-\mu_0\|_{BL^*}\right].
	\end{split}
	\end{equation}
In order to bound the right-hand side of \eqref{eq:equicontinuity_one_time_step} uniformly, we still have to show that the term $\|\mu_{t_l}^N-\mu_0\|_{BL^*}$ in \eqref{eq:equicontinuity_one_time_step} is bounded independent of $N$ and $l$. To see this, note that by Proposition \ref{prop:adaptation of Prop 19} for $N$ large enough $\|\mu_0^N-\mu_0\|_{BL^*}\leq\|\mu_0\|_{BL^*}\Delta_N$. Thus, there exists a constant $C$ such that for all $N\in \mathbb{N}$  $\|\mu_0^N-\mu_0\|_{BL^*}\leq C.$ Similarly to \cite{MR4026977}, we prove the following estimate
 	\begin{align}
 	\label{eq:estimate_induction}
 	\|\mu_{t_l}^N-\mu_0\|_{BL^*}\leq (1+C\Delta_N)^lC+(1+C\Delta_N)^l-1
 	\end{align}
by induction over $l$. As the case for $l=0$ is already proven, we assume that \eqref{eq:estimate_induction} holds for some $l\in \mathbb{N}_0$. Applying \eqref{eq:equicontinuity_one_time_step} with $\tau_1=0$ and $\tau_2=\Delta_N$ yields
	\begin{align*}
	\|\mu_{t_{l+1}}^N-\mu_0\|_{BL^*}&\leq\|\mu_{t_{l+1}}^N-\mu_{t_l}^N\|_{BL^*}+\|\mu_{t_{l}}^N-\mu_0\|_{BL^*}\\
	&\leq \Delta_N\left[C+C\|\mu_{t_{l}}^N-\mu_0\|_{BL^*}\right]+\|\mu_{t_{l}}^N-\mu_0\|_{BL^*}\\
	&=(1+C\Delta_N)\|\mu_{t_{l}}^N-\mu_0\|_{BL^*}+C\Delta_N\\
	&\leq (1+C\Delta_N)\left[(1+C\Delta_N)^lC+(1+C\Delta_N)^l-1\right]+C\Delta_N\\
	&= (1+C\Delta_N)^{l+1}C+(1+C\Delta_N)^{l+1}-1,
	\end{align*}
proving \eqref{eq:estimate_induction}. Hence, as $l\Delta_N\leq T$, we have
	\begin{align}
	\label{eq:auxilliary bound}
	\|\mu_{t_l}^N-\mu_0\|_{BL^*}\leq Ce^{CT}+e^{CT}-1<\infty.
	\end{align}
Plugging  \eqref{eq:auxilliary bound} into \eqref{eq:equicontinuity_one_time_step} yields
	\begin{align*}
	\|\mu_{t_l+\tau_1}^N-\mu_{t_l+\tau_2}^N\|_{BL^*}\leq C|\tau_1-\tau_2|,
	\end{align*}
and by a series of triangle inequalities we also get  for arbitrary $s,t\in [0,T]$
	\begin{align}
	\label{eq:uniiformly lipschitz with s}
	\|\mu_t^N-\mu_s^N\|_{BL^*}\leq C|t-s|.
	\end{align}
In particular, $(\mu_t^N)_{N\in \mathbb{N}}$ is uniformly Lipschitz continuous with respect to $t$ as the Lipschitz constant is independent of $N$. Furthermore, we see that for all $N\in \mathbb{N}$ and all $t\in [0,T]$
    \begin{align}
    \label{eq:uniform boundedness of solution}
	\|\mu_{t}^N\|_{BL^*}\leq \|\mu_{t}^N-\mu_0\|_{BL^*}+\|\mu_0\|_{BL^*}\leq CT+\|\mu_0\|_{BL^*},
	\end{align}
i.e. $(\mu_{t}^N)_{N\in \mathbb{N}}$ is uniformly bounded. Now, the first part of \eqref{eq:uniform_bounds_TV_sum} follows from Lemma \ref{lem:connection_mu_Vmu} and the second by a combination of Lemma \ref{lem:properties_scheme_new}  and \eqref{eq:uniform boundedness of solution}
    \begin{align*}
        \sum_{i=1}^I\sum_{j=1}^J  m_{i,j}^v(V[\mu^N_{t_l}]) \,(1+|v_j| + |v_j|^2)\leq C\|V[\mu_{t_l}^N]\|_{BL^*}=C\|\mu_{t_l}^N\|_{BL^*}\leq C_d.
    \end{align*}
\end{proof}

\begin{corollary}
\label{cor:boundedness of s mu tl N}
The term $s[\mu_{t_l}^N]$ is bounded in the flat norm by a constant which is independent of $l$ and $N$.
\end{corollary}

\begin{proof}
From \eqref{ass:Lipschitz cont s} and estimate \eqref{eq:auxilliary bound} we see
    \begin{align*}
        \|s[\mu_{t_l}^N]\|_{BL^*}\leq \|s[\mu_{t_l}^N]-s[\mu_0]\|_{BL^*}+\|s[\mu_0]\|_{BL^*}\leq L\|\mu_{t_l}^N-\mu_0\|_{BL^*}+C\leq C.
    \end{align*}
\end{proof}

\section{Proof of the existence result}\label{sec:proof_existence}
\noindent In this section we focus on the existence result formulated in Theorem \ref{thm:existence_with_c}.\\

\begin{proofof}{Theorem \ref{thm:existence_with_c}}
Let $\mu_t^N$ be the lattice approximate solution constructed in \eqref{eq:approx_scheme_extended with s}. We start by extracting a converging subsequence with the theorem of Arzel\`a-Ascoli. Combining the results on Lipschitz continuity and uniform boundedness of Lemma \ref{lem:lipschitz_cont_bounds} with the considerations on pointwise compactness in Lemma \ref{lemma:pointwise_compactness}, all requirements of Theorem \ref{thm:arzela_ascoli} (Arzel\`a-Ascoli) are fulfilled. Consequently, the sequence $(\mu_t^N)_{N \in \N}$ has a subsequence (still denoted by $(\mu_t^N)_{N\in \mathbb{N}}$) converging in the space
\begin{align*}
E := C([0,T]; (\mathcal{M}^+(\R^d), \| \cdot \|_{BL^*}))
\end{align*}
 with limit measure $\mu_t$. We claim that $\mu_t$ solves \eqref{eq:MDE_with_source_c}, i.e. it satisfies Definition \ref{def:measure_solution}. To prove this, we fix $t \in [0,T]$ and $f\in C_c^{\infty}(\R^d)$. Introducing the notation $s_1 \wedge s_2 := \mbox{min}(s_1,s_2)$ we write
\begin{align*}
\int_{\R^d} f(x) \diff (\mu^N_t - \mu^N_0)(x) = \sum_{l=0}^{M} \int_{\R^d} f(x) \diff (\mu^N_{t_{l+1} \wedge t} - \mu^N_{t_l \wedge t})(x)
\end{align*}
and we want to study each summand separately. Let $\tau_l =(t_{l+1} \wedge t) - (t_{l} \wedge t) \in [0,\Delta_N]$, then representation \eqref{eq:approx_scheme_extended with s} implies
\begin{equation}
\label{eq:study_convergence}
\begin{split}
&\int_{\R^d} f(x) \diff (\mu^N_{t_{l+1} \wedge t} - \mu^N_{t_{l} \wedge t})(x)  \\
&\qquad \qquad
=\phantom{:} \tau_l \sum_{i=1}^Im_i^x(s[\mu_{t_l}^N])f(x_i)+\sum_{i=1}^I \sum_{j=1}^J m_{i,j}^v\left(V[\mu_{t_l}^N]\right)\, \left[ f(x_i + \tau_l v_j)\, e^{c(x_i,\mu_{t_l}^N)\,\tau_l} - f(x_i) \right]\\
&\qquad \qquad
=\phantom{:} \tau_l \sum_{i=1}^Im_i^x(s[\mu_{t_l}^N])f(x_i)+\sum_{i=1}^I \sum_{j=1}^J m_{i,j}^v\left(V[\mu_{t_l}^N]\right)\, f(x_i + \tau_l v_j)\, \left[  e^{c(x_i,\mu_{t_l}^N)\,\tau_l} - 1 \right] \\
&\qquad \qquad
\phantom{ = } \, + \sum_{i=1}^I \sum_{j=1}^J m_{i,j}^v\left(V[\mu_{t_l}^N]\right)\, \left[ f(x_i + \tau_l v_j) - f(x_i) \right] \\
&\qquad \qquad
=:\tau_l \sum_{i=1}^Im_i^x(s[\mu_{t_l}^N])f(x_i)+\sum_{i=1}^I \sum_{j=1}^J A_{i,j,l}+ \sum_{i=1}^I \sum_{j=1}^J B_{i,j,l}.
\end{split}
\end{equation}
\noindent \underline{Terms $B_{i,j,l}$.} We claim that
\begin{equation}\label{eq:claim_for_Bijl}
\sum_{l=0}^{M}\sum_{i=1}^I \sum_{j=1}^J B_{i,j,l}   \to \int_0^t\int_{\R^d\times \R^d} \nabla f(x) \cdot v \diff V[\mu_t](x,v) \diff t\qquad (N\to \infty).
\end{equation}
Indeed, as $f \in C_c^{\infty}(\R^d)$, Taylor's expansion implies that
$$
\left| f(x_i + \tau_l \, v_j) - f(x_i) - \tau_l \, \nabla f(x_i) \cdot v_j \right| \leq C(f)\, \tau^2_l \, |v_j|^2 \leq C(f)\, \Delta_N^2 \, |v_j|^2.
$$
Therefore, we can replace $ B_{i,j,l}$ with
$$
B'_{i,j,l} :=  m_{i,j}^v\left(V[\mu_{t_l}^N]\right)\, \tau_l\, \nabla f(x_i) \cdot v_j
$$ 
and the error is controlled by
$$
\sum_{l=0}^{M}\sum_{i=1}^I \sum_{j=1}^J \left|B_{i,j,l} - B'_{i,j,l}\right| \leq \sum_{l=0}^{M}\sum_{i=1}^I \sum_{j=1}^J m_{i,j}^v\left(V[\mu_{t_l}^N]\right)\, C(f)\, \Delta_N^2 \, |v_j|^2 \leq C_d \, C(f)\, \Delta_N \,T
$$
thanks to \eqref{eq:uniform_bounds_TV_sum} and $\sum_{l=0}^{M} \Delta_N \leq T$. Now, for $B'_{i,j,l}$ we have
    \begin {align}
    \label{auxiliary equation for Bijl}
\sum_{l=0}^{M}\sum_{i=1}^I \sum_{j=1}^J B'_{i,j,l} = \sum_{l=0}^{M}\int_{t_{l} \wedge t}^{t_{l+1} \wedge t}\int_{\R^d\times \R^d}\nabla f(x)\cdot v\diff  \mathcal{A}^v_N(V[\mu_{t_l}^N])(x,v)\diff r.
    \end{align}
Note that the map $(x,v) \mapsto \nabla f(x) \cdot v$ may be assumed to be in $BL(\R^d \times \R^d)$ thanks to \eqref{ass:support of velocity} and Corollary \ref{cor: compact support of V mu N}. Hence, according to Proposition \ref{prop:adaptation of Prop 19} we can replace the right-hand side of \eqref{auxiliary equation for Bijl} by    
    \begin{align}
        \label{another auxiliary equation for Bijl}
        \sum_{l=0}^{M} \int_{t_{l} \wedge t}^{t_{l+1} \wedge t} \int_{\R^d\times \R^d}\nabla f(x) \cdot v \diff V[\mu_{t_l}^N](x,v) \diff r
    \end{align}
and the error is controlled by 
    \begin{multline*}
         \sum_{l=0}^{M} \int_{t_{l} \wedge t}^{t_{l+1} \wedge t} \int_{\R^d\times \R^d}\nabla f(x) \cdot v \diff\left[\mathcal{A}^v_N(V[\mu_{t_l}^N])- V[\mu_{t_l}^N]\right](x,v) \diff r\\
         \leq \sum_{l=0}^{M}\tau_l\|\nabla f(x)\cdot v\|_{BL}\|\mathcal{A}^v_N(V[\mu_{t_l}^N])-V[\mu_{t_l}^N]\|_{BL^*} 
         \leq  CT\Delta_N\|\nabla f(x)\cdot v\|_{BL}\|\mu_{t_l}^N\|_{BL^*}\to 0
    \end{multline*}
as $N \to \infty$ because $\|\mu_{t_l}^N\|_{BL^*}$ is uniformly bounded by \eqref{eq:uniform_bounds_TV_sum}.\\

\noindent Lastly, we want to replace the measure $V[\mu_{t_l}^N]$ in \eqref{another auxiliary equation for Bijl} by $V[\mu_{r}^N]$ for an arbitrary time point $r\in [t_l\wedge t, t_{l+1}\wedge t] $. Therefore, using \eqref{ass:Lipschitz cont V} and Lipschitz continuity from \eqref{eq:uniiformly lipschitz with s} we obtain
\begin{align*}
&\sum_{l=0}^{M} \int_{t_{l} \wedge t}^{t_{l+1} \wedge t}\int_{\R^d\times \R^d} \nabla f(x) \cdot v \diff\! \left(V[\mu_{t_l}^N] -V[\mu_{r}^N] \right)(x,v) \diff r  \\ 
& \qquad \qquad \qquad \leq \| \nabla f(x) \cdot v \|_{BL} \, \sum_{l=0}^{M} \int_{t_{l} \wedge t}^{t_{l+1} \wedge t} \|V[\mu_{t_l}^N] -V[\mu_{r}^N] \|_{BL^*}   \diff r\\
& \qquad \qquad \qquad \leq \| \nabla f(x) \cdot v \|_{BL} \, C_F \sum_{l=0}^{M} \int_{t_{l} \wedge t}^{t_{l+1} \wedge t} \| \mu_{t_l}^N - \mu_{r}^N \|_{BL^*} \diff r\\
& \qquad \qquad \qquad \leq\| \nabla f(x) \cdot v \|_{BL} \, C \sum_{l=0}^{M} \int_{t_{l} \wedge t}^{t_{l+1} \wedge t} |r-t_l| \diff r\\
& \qquad \qquad \qquad \leq \Delta_N \, \| \nabla f(x) \cdot v \|_{BL} \, C \, T,
\end{align*}
where we used the following inequality in the last line
    \begin{align}
    \label{integral r-tl estimate}
        \sum_{l=0}^{M} \int_{t_{l} \wedge t}^{t_{l+1} \wedge t} |r-t_l| \diff r\leq \Delta_N T.  
    \end{align}
Indeed, this is true as we note that 
    \begin{align}
        \label{eq:auxiliary estimate delta n}
       & |r-t_l|\leq \Delta_N\qquad \forall r\in[t_l\wedge t, t_{l+1}\wedge t].
    \end{align}
Then claim \eqref{integral r-tl estimate} follows  by combining inequality \eqref{eq:auxiliary estimate delta n} with the estimates $(t_{l+1}\wedge t)-(t_l\wedge t)\leq \Delta_N$ and $\sum_{l=1}^{M}\Delta_N\leq T$.
\noindent To sum up, up to an error of order $\Delta_N$, we have
$$
\sum_{l=0}^{M}\sum_{i=1}^I \sum_{j=1}^J B_{i,j,l}   \approx \sum_{l=0}^{M} \int_{t_{l} \wedge t}^{t_{l+1} \wedge t} \int_{\R^d} \int_{\R^d} \nabla f(x) \cdot v \diff V[\mu_r^N](x,v) \diff r\qquad (N\to \infty).
$$
We conclude by seeing
$$
\sum_{l=0}^{M} \int_{t_{l} \wedge t}^{t_{l+1} \wedge t} \int_{\R^d\times \R^d}  \nabla f(x) \cdot v \diff V[\mu_r^N](x,v) \diff r = \int_0^{t}  \int_{\R^d\times\R^d} \nabla f(x) \cdot v \diff V[\mu_r^N](x,v) \diff r
$$
and use $\mu_r^N \to \mu_r$ in $E$ together with assumption \eqref{ass:Lipschitz cont V} to deduce \eqref{eq:claim_for_Bijl}.\\

\noindent {\underline{Terms $A_{i,j,l}$.}} We claim that
\begin{equation}\label{eq:claim_for_Aijl}
\sum_{l=0}^{M} \sum_{i=1}^I \sum_{j=1}^J A_{i,j,l} \to \int_0^t \int_{\R^d} f(x)\,c(x,\mu_r) \diff \mu_{r}(x) \diff r
\qquad (N\to \infty).
\end{equation}
Note simple Taylor's estimates
\begin{equation}\label{eq:estimates_on_exp}
\left|e^{c(x_i,\mu_{t_l}^N)\tau_l} - 1\right| \leq e^{\|c\|_{\infty}\Delta_N} \|c\|_{\infty}  \tau_l, \qquad \left|e^{c(x_i,\mu_{t_l}^N)\tau_l} - 1 - c(x_i,\mu_{t_l}^N)\tau_l \right| \leq  e^{\|c\|_{\infty}\Delta_N} \|c\|_{\infty}^2  \tau^2_l.
\end{equation}
Thanks to the first estimate, the term $A_{i,j,l}$ can be replaced by 
\begin{align*}
A'_{i,j,l} := m_{i,j}^v\left(V[\mu_{t_l}^N]\right)\, f(x_i)\, \left[  e^{c(x_i,\mu_{t_l}^N)\,\tau_l} - 1 \right]
\end{align*}
and the total error is controlled by
\begin{align*}
\sum_{l=0}^{M} \sum_{i=1}^I \sum_{j=1}^J|A_{i,j,l}-A'_{i,j,l}| \leq & \sum_{l=0}^{M}\sum_{i=1}^I \sum_{j=1}^Jm_{i,j}^v\left(V[\mu_{t_l}^N]\right)\, |f|_{\Lip} \, \tau_l \, |v_j| \, e^{\|c\|_{\infty}\Delta_N} \|c\|_{\infty}  \tau_l \\
\leq & C(c,f)\, \sum_{l=0}^{M}\sum_{i=1}^I \sum_{j=1}^J\,m_{i,j}^v\left(V[\mu_{t_l}^N]\right)\,|v_j| \,(\Delta_N)^2 \leq C(c,f)\, C_d\, T\, \Delta_N
\end{align*}
due to \eqref{eq:uniform_bounds_TV_sum} and $\sum_{l=0}^{M} \Delta_N \leq T$. Similarly, term $A'_{i,j,l}$ can now be replaced by 
\begin{align*}
A''_{i,j,l}:=  m_{i,j}^v\left(V[\mu_{t_l}^N]\right)\, f(x_i)\, c(x_i,\mu_{t_l}^N) \, \tau_l
\end{align*}
and this time we use the second estimate in \eqref{eq:estimates_on_exp} for controlling the error. From the definition of $\mu_{t_l}^N$ we obtain
\begin{align*}
\sum_{i=1}^I \sum_{j=1}^J A''_{i,j,l} = \tau_l \int_{\R^d} f(x)\,c(x,\mu_{t_l}^N) \diff \mu_{t_l}^N(x) = \int_{t_l \wedge t}^{t_{l+1}\wedge t} \int_{\R^d} f(x)\,c(x,\mu_{t_l}^N) \diff \mu_{t_l}^N(x) \diff r.
\end{align*}
This can be replaced with $\int_{t_l \wedge t}^{t_{l+1}\wedge t} \int_{\R^d} f(x)\,c(x,\mu_r^N) \diff \mu_{r}^N(x) \diff r$ and the total error can be estimated by
\begin{align*}
&\sum_{l=0}^{M} \left|\int_{t_l \wedge t}^{t_{l+1}\wedge t} \int_{\R^d} f(x)\,(c(x,\mu_{t_l}^N)-c(x,\mu_r^N)) \diff \mu_{t_l}^N(x)\diff r \right|\\
&\hspace{4cm} +\sum_{l=0}^{M} \left|\int_{t_l \wedge t}^{t_{l+1}\wedge t} \int_{\R^d} f(x)\,c(x,\mu_r^N) \diff (\mu_{t_l}^N - \mu_{r}^N)(x)\diff r \right|. 
\end{align*}
For the first term we note that
\begin{multline*}
    \sum_{l=0}^{M} \left|\int_{t_l \wedge t}^{t_{l+1}\wedge t} \int_{\R^d} f(x)\,(c(x,\mu_{t_l}^N)-c(x,\mu_r^N)) \diff \mu_{t_l}^N(x)\diff r \right|  \\
    \leq \| f\|_{BL} \, \sum_{l=0}^{M} \int_{t_l \wedge t}^{t_{l+1}\wedge t}C_L\|\mu_{t_l}^N-\mu_r^N\|_{BL^*} \| \mu_{t_l}^N \|_{BL^*} \diff r
    \leq C_d^2\,C_L \|f\|_{BL} \, \Delta_{N}\,T,
\end{multline*}
while the second term can be bounded by
\begin{multline*}
   \sum_{l=0}^{M} \left|\int_{t_l \wedge t}^{t_{l+1}\wedge t} \int_{\R^d} f(x)\,c(x,\mu_r^N) \diff (\mu_{t_l}^N - \mu_{r}^N)(x)\diff r \right|\\
   \leq\| f \, c\|_{BL}\sum_{l=0}^{M} \int_{t_l \wedge t}^{t_{l+1}\wedge t}\|\mu_{t_l}^N-\mu_r^N\|_{BL^*} \diff r
    \leq \| f \, c\|_{BL}\,C_d\,\Delta_N \, T
\end{multline*}
so that the total error is controlled by $C(c,f)\,\Delta_N\,T$. Note that we used \eqref{integral r-tl estimate} and Lemma \ref{lem:lipschitz_cont_bounds} again. Noting that
    \begin{align*}
       \sum_{l=0}^{M} \int_{t_l \wedge t}^{t_{l+1}\wedge t} \int_{\R^d} f(x)\,c(x, \mu_r^N) \diff  \mu_{r}^N(x) \diff r= \int_{0 }^{t} \int_{\R^d} f(x)\,c(x, \mu_r^N) \diff  \mu_{r}^N(x) \diff r
    \end{align*}
and $\mu_r^N \to \mu_r$ in $E$ we prove \eqref{eq:claim_for_Aijl}.\\

\noindent {\underline{Terms $\tau_l \sum_{i=1}^Im_i^x(s[\mu_{t_l}^N])f(x_i)$.}} We claim that
\begin{equation}\label{eq:claim_for_smu}
\sum_{l=0}^{M}\tau_l \sum_{i=1}^Im_i^x(s[\mu_{t_l}^N])f(x_i) \to \int_0^t\int_{\R^d}f(x)\diff s[\mu_r](x)\diff r
\qquad (N\to \infty).
\end{equation}
First, observe that 
	\begin{equation*}
	\begin{split}
	\sum_{l=0}^{M}&\tau_l \sum_{i=1}^Im_i^x(s[\mu_{t_l}^N])f(x_i)-\int_0^t\int_{\R^d}f(x)\diff s[\mu_r](x)\diff r \\
	 =& \left(\sum_{l=0}^{M}\tau_l \sum_{i=1}^Im_i^x(s[\mu_{t_l}^N])f(x_i)-\int_0^t\int_{\R^d}f(x)\diff s[\mu_r^N](x)\diff r\right)  +\int_0^t\int_{\R^d}f(x)\diff \big( s[\mu_r^N]-s[\mu_r]\big)(x)\diff r\\
	 =&: X + Y. \phantom{\sum_{l=0}^{M}}
	\end{split}
	\end{equation*}
The convergence of the second summand to zero as $N \to \infty$ can be seen directly via
	\begin{equation}
	\label{eq:second term estimate}
	\begin{split}
	|Y| &= \left|\int_0^t\int_{\R^d}f(x)\diff \big( s[\mu_r^N]-s[\mu_r]\big)(x)\diff r\right| \\ 
	&\leq \int_0^T\|f\|_{BL}\|s[\mu_r^N]-s[\mu_r]\|_{BL^*}\diff r 
	\leq T\|f\|_{BL}L\sup_{r\in [0,T]}\|\mu_r^N-\mu_r\|_{BL^*}\to 0,
	\end{split}
	\end{equation}
where we used Assumption \eqref{ass:Lipschitz cont s} and that $\mu_r^N\to \mu_r$ in $E$ with respect to the supremum norm as $N\to \infty$. For term $X$, we observe that
$$
\sum_{l=0}^{M}\tau_l \sum_{i=1}^Im_i^x(s[\mu_{t_l}^N])f(x_i) = \sum_{l=0}^{M} \int_{t_l \wedge t}^{t_{l+1}\wedge t} \int_{\R^d} f(x) \diff \mathcal{S}^N[\mu_{t_l}^N](x) \diff r, \qquad \mathcal{S}^N[\mu_{t_l}^N] := \sum_{i=1}^{I} m_i^x(s[\mu_{t_l}^N]) \, \delta_{x_i},
$$
so that
\begin{equation}\label{termX}
\begin{split}
|X| \leq \sum_{l=0}^{M} \int_{t_l \wedge t}^{t_{l+1}\wedge t} \int_{\R^d} f(x) &\diff\,\left( \mathcal{S}^N[\mu_{t_l}^N] - s[\mu_{r}^N] \right)(x) \diff r  \\ 
& \qquad \qquad \leq 
\|f\|_{BL} \sum_{l=0}^{M} \int_{t_l \wedge t}^{t_{l+1}\wedge t} \| \mathcal{S}^N[\mu_{t_l}^N] - s[\mu_{r}^N]\|_{BL^*} \diff r.
\end{split}
\end{equation}

\noindent Using \eqref{ass:Lipschitz cont s}, Proposition \ref{prop:adaptation of Prop 19} and \eqref{eq:uniiformly lipschitz with s}, we see 
	\begin{equation}
	\label{eq:estimate on A_N}
	\begin{split}
	\| \mathcal{S}^N[\mu_{t_l}^N]& - s[\mu_{r}^N]\|_{BL^*} \leq \,  \|\mathcal{S}^N[\mu_{t_l}^N]-s[\mu_{t_l}^N]\|_{BL^*}+\|s[\mu_{t_l}^N]-s[\mu_{r}^N]\|_{BL^*}\\
	\leq &\sqrt{d}\,\Delta_N^2\,\|s[\mu_{t_l}^N]\|_{BL^*}+L\|\mu_{t_l}^N-\mu_{r}^N\|_{BL^*}\leq \sqrt{d} \, \Delta_N^2 \, \|s[\mu_{t_l}^N]\|_{BL^*}+LC\Delta_N.
	\end{split}
	\end{equation}
Again using \eqref{ass:Lipschitz cont s}, \eqref{eq:uniiformly lipschitz with s} and Proposition \ref{prop:adaptation of Prop 19} we note that
	\begin{equation}
	\label{eq:some_autxilliary_estimate on s[mu]}
	\begin{split}
	\|s[\mu_{t_l}^N]\|_{BL^*}\leq& \, \|s[\mu_0]\|_{BL^*}+\|s[\mu_{t_l}^N]-s[\mu_0]\|_{BL^*}\\
	\leq& \,\|s[\mu_0]\|_{BL^*}+L\|\mu_{t_l}^N-\mu_0\|_{BL^*}\\
	\leq& \,\|s[\mu_0]\|_{BL^*}+L\left[\|\mu_{t_l}^N-\mu_0^N\|_{BL^*}+\|\mu_{0}^N-\mu_0\|_{BL^*}\right]\\
	\leq&\, \|s[\mu_0]\|_{BL^*}+L\left[C|t_l-0|+\Delta_N^2\|\mu_0\|_{BL^*}\right]\\
	=& \,\|s[\mu_0]\|_{BL^*}+LCT+L\Delta_N^2\|\mu_0\|_{BL^*}.
	\end{split}
	\end{equation}
Plugging \eqref{eq:some_autxilliary_estimate on s[mu]} into \eqref{eq:estimate on A_N} and then into \eqref{termX} leads to
$$
|X| \leq C\, T\, \Delta_N \to 0 \qquad (N \to \infty).
$$
\end{proofof}

\begin{remark}
Let $\mu_{\bullet}$ be a solution to problem \eqref{eq:MDE_with_source_c} according to Theorem \ref{thm:existence_with_c}. Then Lemma \ref{lemma:boundedness of support} implies that $\mu_t$ has a compact support $K$ which is independent of $t\in [0,T]$.
\end{remark}

\section{Continuity with respect to initial data}
\label{section:continuity}

\noindent This section is devoted to the continuity of the semigroup characterized by Theorem  \ref{theorem: dependence initial data with new assumption}. Building the proof on the new MVF continuity condition \eqref{eq:new_ass_BL}, we explore its applicability in examples proposed previously as the MDE test cases by Piccoli \cite{MR3961299}.\\

\begin{proof}[Proof of Theorem  \ref{theorem: dependence initial data with new assumption}]
Let $\tau\in [0,\Delta_N]$ with $N$ big enough and consider the lattice approximate solutions $\mu_{t_l + \tau}^N, \nu_{t_l + \tau}^N$. In a first step, we want to estimate the difference $\|\mu_{t_l + \tau}^N - \nu_{t_l + \tau}^N \|_{BL^*}$ by means of $\|\mu_{t_l}^N - \nu_{t_l}^N \|_{BL^*}$. By construction \eqref{eq:approx_scheme_extended with s} it holds
$$
\mu_{t_l + \tau}^N =\tau \sum_{i=1}^Im_i^x(s[\mu_{t_l}^N])\delta_{x_i}+  \sum_{i=1}^I \sum_{j=1}^J m_{i,j}^v\left(V[\mu_{t_l}^N]\right)\, \delta_{x_i + \tau v_j} \, e^{c(x_i, \mu_{t_l}^N)}.
$$
Now, up to an error of size $\Delta_N^2$, we can replace $\mu_{t_l + \tau}^N$ with
    \begin{align}
    \label{approx mu_tlN}
 \tau \sum_{i=1}^Im_i^x(s[\mu_{t_l}^N])\delta_{x_i}+  \sum_{i=1}^I \sum_{j=1}^J m_{i,j}^v\left(V[\mu_{t_l}^N]\right)\, \delta_{x_i + \tau v_j} \, c(x_i, \mu_{t_l}^N) \, \tau + \sum_{i=1}^I \sum_{j=1}^J m_{i,j}^v\left(V[\mu_{t_l}^N]\right)\, \delta_{x_i + \tau v_j} 
    \end{align}
in the flat norm
which follows from a Taylor estimate similar to the second one in \eqref{eq:estimates_on_exp}. Analogously, by paying with an error of order $\Delta_N^2$ we replace $\nu_{t_l + \tau}^N$ with 
    \begin{align}
        \label{approx nu_tlN}
 \tau \sum_{i=1}^Im_i^x(s[\nu_{t_l}^N])\delta_{x_i}+  \sum_{i=1}^I \sum_{j=1}^J m_{i,j}^v\left(V[\nu_{t_l}^N]\right)\, \delta_{x_i + \tau v_j} \, c(x_i, \nu_{t_l}^N) \, \tau + \sum_{i=1}^I \sum_{j=1}^J m_{i,j}^v\left(V[\nu_{t_l}^N]\right)\, \delta_{x_i + \tau v_j}. 
    \end{align}
Now, we estimate the difference between \eqref{approx mu_tlN} and \eqref{approx nu_tlN} in the flat norm by comparing the related terms separately.\\

\noindent \underline{\it Term with source $s$.} Concerning the source term, we have
\begin{multline*}
\left\| \tau \sum_{i=1}^Im_i^x(s[\mu_{t_l}^N])\delta_{x_i} -  \tau \sum_{i=1}^Im_i^x(s[\nu_{t_l}^N])\delta_{x_i} \right\|_{BL^*} 
= \tau \left\| \mathcal{A}^x_N(s[\mu_{t_l}^N]) - \mathcal{A}^x_N(s[\nu_{t_l}^N]) \right\|_{BL^*}  \\
{\phantom{\left\| \tau \sum_{i=1}^I \right\|}} \leq
\tau \left\| \mathcal{A}^x_N(s[\mu_{t_l}^N]) - s[\mu_{t_l}^N] \right\|_{BL^*}
+ \tau\, \left\|s[\mu_{t_l}^N] - s[\nu_{t_l}^N] \right\|_{BL^*}
+ \tau \left\| \mathcal{A}^x_N(s[\nu_{t_l}^N]) - s[\nu_{t_l}^N] \right\|_{BL^*}.
\end{multline*}
Using Proposition \ref{prop:adaptation of Prop 19}, Corollary \ref{cor:boundedness of s mu tl N} and assumption \eqref{ass:Lipschitz cont s}, we obtain
    \begin{align}
    \label{ineq:term s}
\left\| \tau \sum_{i=1}^Im_i^x(s[\mu_{t_l}^N])\delta_{x_i} -  \tau \sum_{i=1}^Im_i^x(s[\nu_{t_l}^N])\delta_{x_i} \right\|_{BL^*} \leq C \Delta_N^3 + L\Delta_N \|\mu_{t_l}^N - \nu_{t_l}^N \|_{BL^*}.
    \end{align}

\noindent \underline{\it Term with growth function $c$.} First, we want to transform this term to a simpler expression. Using $\|\delta_a-\delta_b\|_{BL^*}\leq |a-b|$, we observe
\begin{multline*}
\left\| \sum_{i=1}^I \sum_{j=1}^J m_{i,j}^v\left(V[\mu_{t_l}^N]\right)\, \delta_{x_i + \tau v_j} \, c(x_i, \mu_{t_l}^N) \, \tau - \sum_{i=1}^I \sum_{j=1}^J m_{i,j}^v\left(V[\mu_{t_l}^N]\right)\, \delta_{x_i} \, c(x_i, \mu_{t_l}^N) \, \tau  \right\|_{BL^*}  \\ 
\leq \tau \, \|c\|_{\infty} \, \sum_{i=1}^I \sum_{j=1}^J m_{i,j}^v\left(V[\mu_{t_l}^N]\right) \left\| \, \delta_{x_i + \tau v_j} - \delta_{x_i} \right\|_{BL^*} \leq
\tau^2 \, \|c\|_{\infty} \,  \sum_{i=1}^I \sum_{j=1}^J m_{i,j}^v\left(V[\mu_{t_l}^N]\right) \, |v_j| \leq C\, \Delta_N^2,
\end{multline*}
where we applied Lemma \ref{lem:lipschitz_cont_bounds} in the last step. Furthermore, as $\pi_1^{\#}V[\mu]=\mu$, we note that 
$$
\sum_{i=1}^I \sum_{j=1}^J m_{i,j}^v\left(V[\mu_{t_l}^N]\right)\, \delta_{x_i} \, c(x_i, \mu_{t_l}^N) \, \tau = 
\tau \sum_{i=1}^I m_i^x(\mu_{t_l}^N) \, \delta_{x_i} \, c(x_i, \mu_{t_l}^N).
$$
The latter measure is an approximation of  $\tau\,\mu_{t_l}^N \, c(\cdot,\mu_{t_l}^N)$. To see that, we consider a test function $\psi \in BL(\R^d)$ with $\|\psi\|_{BL}\leq 1$ and compute
\begin{align*}
&\tau \left|\int_{\R^d} \psi(x) \, c(x,\mu_{t_l}^N) \diff \mu_{t_l}^N -
\sum_{i=1}^I \psi(x_i)  \, c(x_i, \mu_{t_l}^N) \, m_i^x(\mu_{t_l}^N) \right|  \\ 
& \qquad \qquad \leq \tau \sum_{i=1}^I \int_{Q_i} \left|\psi(x) \, c(x,\mu_{t_l}^N) -\psi(x_i)  \, c(x_i, \mu_{t_l}^N) \,\right| \diff \mu_{t_l}^N \\
& \qquad \qquad \leq\tau\sum_{i=1}^I \int_{Q_i} |\psi(x)-\psi(x_i)| \, |c(x,\mu_{t_l}^N)|+ |\psi(x_i)|  \, |c(x,\mu_{t_l}^N)-c(x_i, \mu_{t_l}^N)| \diff \mu_{t_l}^N \leq C\,\Delta_N^3. 
\end{align*}
Note that we used assumption \eqref{ass:Lipschitz cont c} and Lemma \ref{lem:lipschitz_cont_bounds}.
We conclude that 
$$
\left\| \sum_{i=1}^I \sum_{j=1}^J m_{i,j}^v\left(V[\mu_{t_l}^N]\right)\, \delta_{x_i + \tau v_j} \, c(x_i, \mu_{t_l}^N) \, \tau - \tau\,\mu_{t_l}^N \, c(\cdot,\mu_{t_l}^N)  \right\|_{BL^*} \leq C\Delta_N^2 + C\Delta_N^3
$$
and similar estimate holds for the expression with $\nu_{t_l}^N$ instead of $\mu_{t_l}^N$. Therefore, to compare terms containing function $c$, it is sufficient to estimate $ \tau\,\mu_{t_l}^N \, c(\cdot,\mu_{t_l}^N) -  \tau\,\nu_{t_l}^N \, c(\cdot,\nu_{t_l}^N)$ which can be bounded by
$$
\left\| \tau\,\nu_{t_l}^N \, c(\cdot,\nu_{t_l}^N) -  \tau\,\mu_{t_l}^N \, c(\cdot,\mu_{t_l}^N) \right\|_{BL^*} \leq C\tau \, \|\nu_{t_l}^N - \mu_{t_l}^N \|_{BL^*}
$$
using assumption \eqref{ass:Lipschitz cont c} and Lemma \ref{lem:lipschitz_cont_bounds}. So in total we get the estimate
\begin{align}
    \label{ineq: term c}
    \begin{split}
    \left\| \sum_{i=1}^I \sum_{j=1}^J m_{i,j}^v\left(V[\mu_{t_l}^N]\right)\, \delta_{x_i + \tau v_j} \, c(x_i, \mu_{t_l}^N) \, \tau - \sum_{i=1}^I \sum_{j=1}^J m_{i,j}^v\left(V[\mu_{t_l}^N]\right)\, \delta_{x_i} \, c(x_i, \mu_{t_l}^N) \, \tau  \right\|_{BL^*} \qquad & \\
     \leq C\Delta_N\|\nu_{t_l}^N - \mu_{t_l}^N \|_{BL^*}+C\Delta_N^2+C\Delta_N^3.& 
    \end{split}
\end{align}
\\
\noindent \underline{\it Transport term.} This is the most difficult term to handle  and the additional assumption \eqref{eq:new_ass_BL} will be needed here. We want to estimate
    \begin{align}
        \label{eq:difference in transport term}
    \gamma[\mu_{t_l}^N,\nu_{t_l}^N] := \sum_{i=1}^I \sum_{j=1}^J \left( m_{ij}^v\left(V[\mu_{t_l}^N]\right)\,  -  m_{ij}^v\left(V[\nu_{t_l}^N]\right) \right)\, \delta_{x_i + \tau v_j},\qquad \Delta:= \left\| \gamma[\mu_{t_l}^N,\nu_{t_l}^N] \right\|_{BL^*}.
    \end{align}
The most important observation is that
$$
\int_{\R^d} \psi(x) \diff \gamma[\mu_{t_l}^N,\nu_{t_l}^N](x) = 
\int_{\R^d \times \R^d} \psi(x + \tau \, v) \diff\!\left(\mathcal{A}^v_N(V[\mu^N_{t_l}]) - \mathcal{A}^v_N(V[\nu^N_{t_l}])\right)(x,v).
$$
Now, we want to replace $\mathcal{A}^v_N(V[\mu^N_{t_l}])$ and $\mathcal{A}^v_N(V[\nu^N_{t_l}])$ by $V[\mu^N_{t_l}]$ and $V[\nu^N_{t_l}]$ respectively. By a reasoning similar to the one in Proposition \ref{prop:adaptation of Prop 19} we have
\begin{equation}\label{eq:estimate_with_x+tauv}
\left| \int_{\R^d \times \R^d} \psi(x + \tau \, v) \diff\!\left(\mathcal{A}^v_N(V[\mu^N_{t_l}]) - V[\mu^N_{t_l}] \right)(x,v) \right| \leq C\, \Delta_N^2.
\end{equation}
Indeed, 
\begin{multline*}
 \int_{\R^d \times \R^d} \psi(x + \tau \, v) \diff\!\left(\mathcal{A}^v_N(V[\mu^N_{t_l}]) - V[\mu^N_{t_l}] \right)(x,v)   \\ = \sum_{i=1}^I \sum_{j=1}^J \int_{Q_i \times Q_j'} 
 \left[\psi(x_i + \tau \, v_j) - \psi(x + \tau \, v)\right]
 \diff V[\mu^N_{t_l}](x,v).
\end{multline*}
Using  the Lipschitz continuity of $\psi$, we have 
\begin{align*}
|\psi(x_i + \tau \, v_j) - \psi(x + \tau \, v)| \leq |x_i - x| + \tau \, |v_j - v| \leq \Delta_N^2,
\end{align*}
so that \eqref{eq:estimate_with_x+tauv} follows as $ V[\mu^N_{t_l}](\R^d \times \R^d)$ is uniformly bounded by Lemma \ref{lem:lipschitz_cont_bounds}. Hence, using assumption in \eqref{eq:new_ass_BL} we deduce a bound for the transport term 
    \begin{align}
    \label{bound transport term new assumption}
    \Delta \leq (1 + C_H\, \tau) \, \left\| \mu^N_{t_l} - \nu^N_{t_l}  \right\|_{BL^*} + C\, \Delta_N^2\leq (1 + C_H\, \Delta_N) \, \left\| \mu^N_{t_l} - \nu^N_{t_l}  \right\|_{BL^*} + C\, \Delta_N^2 .
    \end{align}

\noindent Combining the above estimates concerning the source term $s$ (bound \eqref{ineq:term s}), the growth term $c$ (bound \eqref{ineq: term c}) and the transport term (bound \eqref{bound transport term new assumption}), we get an estimate of the form
\begin{align}
\label{eq:final_estimate_t_l+1}
\|\mu^N_{t_{l}+\tau}-\nu^N_{t_{l}+\tau}\|_{BL^*}\leq (1+K\Delta_N)\|\mu^N_{t_l}-\nu^N_{t_l}\|_{BL^*}+C\Delta_N^2
\end{align}
which for $\tau=\Delta_N$ can be iterated to get
    \begin{align*}
	\|\mu^N_{t_l}-\nu^N_{t_l}\|_{BL^*}\leq& (1+K\Delta_N)^l\|\mu^N_0-\nu^N_0\|_{BL^*}+C\Delta_N^2\sum_{r=0}^{l-1}(1+K\Delta_N)^r\\
	\leq& e^{l\Delta_NK}\|\mu^N_0-\nu^N_0\|_{BL^*}+C\Delta_N^2 \frac{(1+K\Delta_N)^l-1}{1+K\Delta_N-1}\\
	\leq& e^{Kt}\|\mu^N_0-\nu^N_0\|_{BL^*}+C\Delta_N^2 \frac{e^{l\Delta_NK}-1}{K\Delta_N}\\
	\leq& e^{Kt}\|\mu^N_0-\nu^N_0\|_{BL^*}+C\Delta_N \frac{e^{Kt}-1}{K},
	\end{align*}
as $l\Delta_N\leq t$. Using triangle inequality with Proposition \ref{prop:adaptation of Prop 19}, we finally get
    \begin{align*}
         \|\mu^N_{t_l}-\nu^N_{t_l}\|_{BL^*}\leq e^{Kt}\|\mu_0-\nu_0\|_{BL^*} +e^{KT}C\Delta_N^2+C\Delta_N\frac {e^{KT}-1}{K}
    \end{align*}
and thus \eqref{eq:dependence initial conditions} follows by letting $N\to \infty$.
\end{proof}

\begin{remark}
In contrast to the proof of the continuity with respect to initial data provided by Piccoli and Rossi (see \cite{MR4026977}), the proof of Theorem \ref{theorem: dependence initial data with new assumption} presented here avoids the technical application of optimal transport theory. This in turn, significantly simplifies the proof.
\end{remark}
\noindent Let us demonstrate that standard examples of MVFs indeed satisfy both conditions \eqref{ass:Lipschitz cont V} and \eqref{eq:new_ass_BL}.
\begin{example}
Let $\textbf{v}(x)$ be a Lipschitz vector field and consider 
$$
V_{1}[\mu] = \mu \otimes \delta_{\textbf{v}(x)}.
$$
Then for any $\psi$ with $\|\psi\|_{BL(\R^d)}\leq 1$
$$
\int_{\R^d \times \R^d} \psi(x + \tau \, v) \diff(V_1[\mu] - V_1[\nu])(x,v) =
\int_{\R^d} \psi(x + \tau\,\textbf{v}(x)) \diff(\mu - \nu)(x).
$$
The map $x \mapsto \psi(x + \tau\,\textbf{v}(x))$ is bounded by 1 and Lipschitz continuous with constant $1 +\tau \, |\textbf{v}|_{\Lip}$. It follows that assumption \eqref{eq:new_ass_BL} is satisfied with $C_H:=|\textbf{v}|_{\Lip}$. Similarly, the MVF $V_1$ also satisfies \eqref{ass:Lipschitz cont V} with $C_F=1+|\textbf{v}|_{\Lip}$ as the map $\tilde\psi:x\mapsto \psi(x,\textbf{v}(x))$ has bound $\|\tilde\psi\|_{BL(\R^d\times \R^d)}\leq 1+|\textbf{v}|_{\Lip}$. So in view of Remark \ref{rem:discussion assumption V2hat for existence} existence and uniqueness of solutions to model \eqref{eq:MDE_with_source_c} with $V=V_1$ can be provided. 
\end{example}

\begin{example}
\label{example:new assumption barycenter example}
 Suppose that $d=1$, $c = 0$ and $s = 0$ so that we consider a 1D conservative MDE, i.e. the map $t \mapsto \int_{\R} \diff \mu_t$ is constant (to see this, consider a compactly supported test function which is identity on the support of the solution $\mu_t$). In such situation, we may consider the MDE in the space of probability measures as in \cite{MR3961299}. Moreover, it is sufficient to obtain the estimates in the Wasserstein distance $W_1$ rather than flat metric so that condition \eqref{eq:new_ass_BL} can be replaced by 
$$
\sup_{ \|\psi\|_{\Lip\left(\R\right)} \leq 1}\int_{\R \times \R} \psi(x + \tau \, v) \diff(V[\mu] - V[\nu])(x,v) \leq (1 + C_H\,\tau) \, W_1(\mu, \nu).
$$

\noindent For $\mu \in \mathcal{P}(\R)$ we define
$$
B(\mu):= \sup\left\{x \in \R: \mu(-\infty, x] \leq \frac{1}{2} \right\}, \qquad b_{\mu}:= \frac{\frac{1}{2} - \mu(-\infty, B(\mu)) }{\mu\{B(\mu)\}},
$$
the latter quantity is only well-defined when $\mu\{B(\mu)\} > 0$. Then the MVF is defined as $V_2[\mu] = \mu \otimes \gamma_x$ where
$$
\gamma_{x} = \begin{cases}
\delta_{-1} &\mbox{ if } x < B(\mu), \\
\delta_{1} &\mbox{ if } x > B(\mu), \\
b_{\mu}\, \delta_{-1} + (1 - b_{\mu}) \, \delta_{1}  &\mbox{ if } x = B(\mu) \mbox{ and } \mu\{B(\mu)\} > 0.
\end{cases}
$$
\end{example}

\noindent To prove that this MVF satisfies the assumption, we will need the following decomposition device concerning Wasserstein distance. Let $\mu \in \mathcal{P}(\R)$. We define equal decomposition of $\mu$ for left and right part
$$
\mu^l := \mu \, \mathds{1}_{(-\infty, B(\mu))} + 
b_{\mu} \, \mu\{B(\mu)\} \, \mathds{1}_{x = B(\mu)},
$$
$$
\mu^r := \mu \, \mathds{1}_{(B(\mu), +\infty)} + (1-b_{\mu}) \, \mu\{B(\mu)\} \, \mathds{1}_{x = B(\mu)}.
$$

\begin{lemma}[decomposition formula for the Wasserstein distance]\label{lem:dec_Wass}
Consider $\mu, \nu \in \mathcal{P}(\R)$. Then
$$
W_1(\mu, \nu) = W_1(\mu^l, \nu^l) + W_1(\mu^r, \nu^r).
$$
\end{lemma}
\begin{remark}
The formula above is another way of expressing the fundamental (and well-known) fact concerning optimal transport in 1D: The optimal maps are always monotone, i.e. they transfer the mass from left to right. 
\end{remark}

\begin{proof}[Proof of Lemma \ref{lem:dec_Wass}]
Let $F$ and $G$ be the cumulative distribution functions (CDF) of $\mu$ and $\nu$ respectively
$$
F(x) = \mu(-\infty, x], \qquad \qquad G(x) = \nu(-\infty, x].
$$
Moreover, let $F^{-1}$ and $G^{-1}$ be their generalized inverses
$$
F^{-1}(y) = \inf\{x: F(x) > y\}, \qquad \qquad G^{-1}(y) = \inf\{x: G(x) > y\}.
$$
It is well-known cf. \cite[Theorem 2.18, Remark 2.19 (ii)]{villani:2003} that
\begin{equation}\label{eq:formula_for_W1}
W_1(\mu, \nu) = \int_0^1 |F^{-1}(y) - G^{-1}(y)| \diff y.
\end{equation}

\noindent First, we consider the term $W_1(\mu^l, \nu^l)$. Note that $2\mu^l$ and $2\nu^l$ are probability measures. Moreover, the CDF of $2\mu^l$ is given by $\tilde F_l:\R\to[0,1]$, where
    \begin{align*}
       \tilde F_l(x)=\begin{cases}
        2F(x)& x<B(\mu)\\
        1& x\geq B(\mu).
        \end{cases}
    \end{align*}
Thus, its generalized inverse is of the form $\tilde F_l^{-1}=F^{-1}(\cdot/2)$ except possibly at $y=1$ which will be negligible as a set of Lebesgue measure zero. Analogously, for $2\nu^l$ we get the CDF $\tilde G_l$ and its generalized inverse $\tilde G_l^{-1}$ by replacing $\mu$ by $\nu$ and $F$ by $G$ in the corresponding function for $2\mu^l$.
Applying formula \eqref{eq:formula_for_W1} to $\mu^l,\nu^l$ we obtain
\begin{align}
\label{eq:wasserstein split left}
\begin{split}
W_1(\mu^l, \nu^l) =& \frac{1}{2}\, W_1(2\mu^l, 2\nu^l) = \frac{1}{2}\,\int_0^1 |\tilde F_l^{-1}(y) - \tilde G_l^{-1}(y)| \diff y\\
=&\frac{1}{2}\,\int_0^1 |F^{-1}(y/2) - G^{-1}(y/2)| \diff y
= \int_0^{1/2} |F^{-1}(y) - G^{-1}(y)| \diff y.
\end{split}
\end{align}
Similarly, we consider the term $W_1(\mu^r, \nu^r)$. This time, the CDF of $2\mu^r$ is
\begin{align*}
    \tilde F_r(x)=\begin{cases}
    0&x<B(\mu)\\
    2(1-b_{\mu})\mu(\{B(\mu)\})& x=B(\mu)\\
    2\left[\mu(-\infty,x]+(1-b_{\mu})\mu(\{B(\mu)\})\right]&x> B(\mu)
    \end{cases} \hspace{0.5cm}=
    \begin{cases}
    0&x<B(\mu)\\
    2(F(x)-\frac 1 2)&x\geq B(\mu).
    \end{cases}
\end{align*} 
The corresponding generalized inverse is thus $\tilde F_r^{-1}=F^{-1}(\cdot/2 + 1/2)$ neglecting possibly the point $y=0$ as a set of Lebesgue measure zero. Analogously, $2\nu^r$ has CDF $\tilde G_r$ with generalized inverse $\tilde G_r^{-1}$.
 We apply \eqref{eq:formula_for_W1} to $2\mu^r,2\nu^r$ and deduce
\begin{align}
\label{eq:wasserstein split right}
\begin{split}
W_1(\mu^r, \nu^r) =& \frac{1}{2}\, W_1(2\mu^r, 2\nu^r)= \frac{1}{2}\,\int_0^1 |\tilde F_r^{-1}(y) - \tilde G_r^{-1}(y)| \diff y    \\ 
=& \frac{1}{2}\,\int_0^1 |F^{-1}(y/2 + 1/2) - G^{-1}(y/2 + 1/2)| \diff y = \int_{1/2}^1 |F^{-1}(y) - G^{-1}(y)| \diff y.
\end{split}
\end{align}
Combining \eqref{eq:wasserstein split left} and \eqref{eq:wasserstein split right} with \eqref{eq:formula_for_W1} yields the desired result.
\end{proof}

\noindent Now, we use Lemma \ref{lem:dec_Wass} to prove Example \ref{example:new assumption barycenter example}, i.e. to see that $V_2$ satisfies assumption \eqref{eq:new_ass_BL}. Indeed, let $\psi \in BL(\R)$ with $\|\psi\|_{BL} \leq 1$. Then,
\begin{align*}
\int_{\R \times \R} \psi(x + \tau \, v) \diff V_2[\mu](x,v) &= \int_{x < B(\mu)} \psi(x - \tau) \diff \mu(x) + \int_{x > B(\mu)} \psi(x + \tau) \diff \mu(x) \\
\phantom{\int_{\R}}
&\phantom{=}+ b_{\mu} \, \mu\{B(\mu)\} \, \psi(B(\mu) - \tau) 
+ (1 - b_{\mu}) \, \mu\{B(\mu)\} \, \psi(B(\mu) + \tau) \\ &= \int_{\R} \psi(x - \tau) \diff \mu^l(x) + \int_{\R} \psi(x + \tau) \diff \mu^r(x).
\end{align*}
It follows that
$$
\int_{\R \times \R} \psi(x + \tau \, v) \diff\! \left(V_2[\mu] - V_2[\nu] \right)(x,v) = 
\int_{\R} \psi(x - \tau) \diff\! \left(\mu^l - \nu^l\right)(x) + \int_{\R} \psi(x + \tau) \diff\!\left(\mu^r - \nu^r\right)(x).
$$
As maps $x \mapsto \psi(x - \tau), \psi(x + \tau)$ are Lipschitz continuous with constant 1, we deduce
$$
\int_{\R \times \R} \psi(x + \tau \, v) \diff \!\left(V_2[\mu] - V_2[\nu] \right)(x,v) \leq W_1(\mu^l, \nu^l) + W_1(\mu^r, \nu^r) = W_1(\mu, \nu).
$$
A similar computation, replacing $\psi(x+\tau v)$ and $\psi(x\pm\tau)$ with $\psi(x,v)$ and $\psi(x,\pm 1)$ respectively, yields that $V_2$ also satisfies assumption \eqref{ass:Lipschitz cont V} with $C_F=1$ where $\|\mu-\nu\|_{BL(\R^d)}$ is replaced by $W_1(\mu,\nu)$. According to \cite{MR3961299}, this is sufficient to guarantee existence in the conservative case.

\begin{remark}
It is not trivial to generalize Example \ref{example:new assumption barycenter example} to the non-conservative case. The main issue is that Lemma \ref{lem:dec_Wass} does not have a natural generalization for flat metric. One could try to prove this with the variational formula for the flat metric (cf. Remark \ref{rem:Piccoli characterisation of flat metric}) but the barycenter of $\mu$ does not tell too much about the barycenter of the submeasure $\widetilde{\mu}$. 
\end{remark}

\section{Lipschitz semigroup of solutions and uniqueness}\label{sect:lip_semigroup_uniqueness}
\noindent We conclude this paper with a short consideration of the uniqueness concept introduced by Piccoli and Rossi \cite{MR4026977}.  We start with the fact that the solutions of \eqref{eq:MDE_with_source_c} form a Lipschitz semigroup in the following sense
\begin{definition}
\label{def:lipschitz semigroup}
A \textbf{Lipschitz semigroup of solutions} $S_t$ to problem \eqref{eq:MDE_with_source_c} is a map $S:[0,T]\times \mathcal{M}^+(\R^d)\to\mathcal{M}^+(\R^d)$ satisfying
\begin{itemize}
    \item[(1)] $S_0\mu_0=\mu_0$ and $S_{t+s}\mu_0=S_tS_s\mu_0$.
    \item [(2)] The map $t\mapsto S_t\mu_0$ is a solution to \eqref{eq:MDE_with_source_c} with initial condition $\mu_0$.
    \item [(3)] For every $R,M>0$ there exists $C=C(R,M)>0$ such that if $\supp(\mu_0)\cup\supp(\nu_0)\subseteq B(0,R)$ and $\|\mu_0+\nu_0\|_{TV}\leq M$, then it holds
        \begin{itemize}
            \item [(a)] $\supp(S_t\mu_0)\subset B(0,e^{Ct}C)$
            \item [(b)] $\|S_t\mu_0-S_t\nu_0\|_{BL^*} \leq e^{Ct}\|\mu_0-\nu_0\|_{BL^*}$
            \item [(c)] $\|S_t\mu_0-S_t\mu_0\|_{BL^*}\leq C|t-s|$.
        \end{itemize}
\end{itemize}
\end{definition}

\begin{corollary}
\label{cor:Lipschitz semigroup solution V wrt Wg}
Under Assumption  \ref{ass:MVF} with \eqref{ass:Lipschitz cont V} supplemented by \eqref{eq:new_ass_BL}, there exists a Lipschitz semigroup of solutions to problem \eqref{eq:MDE_with_source_c}.
\end{corollary}

\begin{proof}
 For $t\in[0,T]$ we define the semigroup as the limit of the lattice approximate solution, i.e.,
    \begin{align*}
    S_t\mu_0:=\mu_t=\lim_{N\to \infty} \mu_t^N.
    \end{align*}
Then properties (1), (2), (3a) and (3c) of Definition \ref{def:lipschitz semigroup} follow directly from the proof of Theorem \ref{thm:existence_with_c} and (3b) follows from Theorem \ref{theorem: dependence initial data with new assumption}.
\end{proof}
\noindent Piccoli and Rossi introduced a uniqueness concept of the Lipschitz semigroup based on given Dirac germs. The proof follows analogously to \cite[Theorem 4]{MR4026977}. For convenience of the reader, we present here the necessary definitions and the result. 

\begin{definition}
\begin{itemize}
    \item[i)] We define the \textbf{positive linear span of Dirac deltas} as
    \begin{align*}
        \mathcal{M}^+_D(\R^d):=\{\mu\in \mathcal{M}^+(\R^d)\mid \mu=\sum_{k=1}^l\alpha_k\delta_{x_l}\}.
    \end{align*}
    \item [(ii)] For constants $R,M>0$ we set
        \begin{align*}
        \mathcal{M}^+_{D;R,M}:=\{\mu\in \mathcal{M}_D^+(\R^d)\mid \supp(\mu)\subseteq B(0,R), \|\mu\|_{TV}\leq M\}.
    \end{align*}
\end{itemize}
\end{definition}

\noindent Next, we introduce the Dirac germs which will be used to obtain a uniqueness result.
\begin{definition}
\label{def:dirac germ and compatible}
Fix a MVF $V$. 
\begin{enumerate}
\item[(i)] A \textbf{Dirac germ $\gamma$  compatible with $V$} is a map assigning to measures in $\mathcal{M}^+_{D;R,M}$ a solution to \eqref{eq:MDE_with_source_c}. More explicitly, for all $R>0$, $M>0$ and all measures $\mu\in \mathcal{M}_{D;R,M}^+(\R^d)$, there exists a constant $\varepsilon(R,M)>0$ and a Lipschitz curve $\gamma_{\mu}:[0,\varepsilon(R,M)]\to\mathcal{M}^+(\R^d)$ solving \eqref{eq:MDE_with_source_c}.
\item[(ii)] Let $\gamma$ be a Dirac germ compatible with $V$. A Lipschitz semigroup of solutions to \eqref{eq:MDE_with_source_c} is called \textbf{compatible with $\gamma$} if for all constants $R,M>0$ there exists a constant $C=C(R,M)$ such that 
    \begin{align*}
      \forall t\in [0,\varepsilon(R,M)]:\qquad  \sup_{\mu\in \mathcal{M}^+_{D;R,M}}\|S_t\mu-\gamma_{\mu}(t)\|_{BL^*}\leq Ct^2.
    \end{align*}
    \end{enumerate}
\end{definition}

\begin{theorem}\cite[Theorem 4]{MR4026977}
\label{thm:uniqueness}
Assume \eqref{ass:support of velocity}, \textbf{(s)} and \textbf{(c)} from Assumption \ref{ass:MVF} hold. Let $\gamma$ be a Dirac germ to model \eqref{eq:MDE_with_source_c} compatible with $V$.Then there exists at most one Lipschitz semigroup compatible with $\gamma$.
\end{theorem}
\begin{proof}
The proof follows the same lines as the proof of Theorem 4 in \cite{MR4026977}.
\end{proof}

\appendix
\section{Alternative MVF continuity condition}
\label{appendix:piccoli proof continuity}
\noindent In this Appendix we prove that results of \cite{MR3961299} and \cite{MR4026977} are in fact special cases of our work. In these papers, Lipschitz continuity of solutions to MDE with respect to initial conditions has been established under the following assumption.
\begin{definition}
\label{def:operator Wg}
Consider two measures $V_1,V_2\in \mathcal{M}^+(\R^d\times \R^d)$ with $\pi_1^{\#}\,V_i=\mu_i$, $i=1,2$. For each pair $(\tilde V_1,\tilde V_2)$ with $\tilde V_1\leq V_1$ and $\tilde V_2\leq V_2$ set $\tilde \mu_i=\pi^{\#}_1\tilde V_i$ and define the operator
    \begin{align}
        \label{def:curly W g}
        \begin{split}
        \mathcal{W}^g(V_1,V_2):=&\inf\bigg\{\int_{(\R^d)^4}|v-w|\diff p(x,v,y,w)\mid \tilde V_1\leq V_1,\tilde V_2\leq V_2, p\in \mathcal{P}(\tilde V_1,\tilde V_2),\\
        &\|\mu_1-\mu_2\|_{BL^*}=\|\mu_1-\tilde\mu_1\|_{TV}+\|\mu_2-\tilde \mu_2\|_{TV}+W_1(\tilde \mu_1,\tilde\mu_2) \text{ and } \pi_{13}^{\#}\,p\in \mathcal{P}^{\text{opt}}(\tilde \mu_1,\tilde \mu_2)\bigg\}.
        \end{split}
    \end{align}
Here, $\mathcal{P}(\tilde V_1,\tilde V_2)$ denotes the set of all transference plans between measures $\tilde V_1$ and $\tilde V_2$. Analogously, $\mathcal{P}^{\text{opt}}(\tilde \mu_1,\tilde \mu_2)$ denotes the set of all optimal transference plans between $\tilde \mu_1$ and $\mu_2$.
\end{definition}

\begin{assumption}\label{ass:Piccoli_old_assumption} For all $R>0$, there is a constant $C_F(R)$ such that if $\mu$, $\nu$ are supported in $B(0,R)$, then
    \begin{align}
    \label{ass:Lipschitz cont V wrt Wg}
    \tag{\textbf{$\widehat{V_{2,3}}$}}
         \mathcal{W}^g(V[\mu],V[\nu])\leq C_F(R)\|\mu-\nu\|_{BL^*}.\hspace{2cm}
    \end{align}
\end{assumption}

\noindent Now, we prove that in fact condition \eqref{ass:Lipschitz cont V wrt Wg} implies conditions \eqref{ass:Lipschitz cont V} and \eqref{eq:new_ass_BL}. 

\begin{lemma}\label{lem:piccoli_implies_ours}
Suppose that condition \eqref{ass:Lipschitz cont V wrt Wg} in Assumption \ref{ass:Piccoli_old_assumption} holds true. Then both \eqref{ass:Lipschitz cont V} in Assumption \ref{ass:MVF} and \eqref{eq:new_ass_BL} in Assumption \ref{ass:new_cont_condition} are satisfied.
\end{lemma}

\begin{proof}
\label{rem:tilde V2 stronger than V2}
Note that in view of Remark \ref{rem:Piccoli characterisation of flat metric} and \cite[Proposition 28]{MR4026977} we have
    \begin{align*}
        \|V[\mu]-V[\nu]\|_{BL^*} \leq \mathcal{W}^g(V[\mu],V[\nu])+\|\mu-\nu\|_{BL^*},
    \end{align*}
so that \eqref{ass:Lipschitz cont V} is satisfied. To see \eqref{eq:new_ass_BL}, we fix $\mu, \nu \in \mathcal{M}^+(\R^d)$ and observe that condition \eqref{ass:Lipschitz cont V wrt Wg} implies that $\mathcal{W}^g(V[\mu],V[\nu]) < \infty$. It is \`{a} priori not clear that the infimum of $\mathcal{W}^g$ is actually attained. But for sure it is almost attained and thus it follows that for all $\varepsilon > 0$, there exist almost optimal submeasures $V_{\mu}^{\varepsilon} \leq V[\mu]$, $V_{\nu}^{\varepsilon} \leq V[\nu]$ and a transference plan $p^{\varepsilon} \in \mathcal{P}(V_{\mu}^{\varepsilon},V_{\nu}^{\varepsilon})$ such that
\begin{itemize}
\item $
\int_{(\R^d)^4}|v-w|\diff p^{\varepsilon}(x,v,x,w) \leq \mathcal{W}^g(V[\mu],V[\nu]) + \varepsilon,
$
\item $
\|\mu-\nu\|_{BL^*}=\|\mu-\mu^{\varepsilon}\|_{TV}+\|\nu-\nu^{\varepsilon}\|_{TV}+W_1(\mu^{\varepsilon},\nu^{\varepsilon}),
$
\item
$
 \pi_{13}^{\#}\,p^{\varepsilon}\in \mathcal{P}^{\text{opt}}(\mu^{\varepsilon}, \nu^{\varepsilon}).
$
\end{itemize}
Now, fix $\tau >0$ and $\psi \in BL(\R^d)$ with $\|\psi\|_{BL} \leq 1$. We compute
\begin{equation}\label{eq:split_proof_pic_implies_ours}
\begin{split}
&\int_{\R^d \times \R^d} \psi(x+\tau v) \diff(V[\mu] - V[\nu])(x,v) = 
\int_{\R^d \times \R^d} \psi(x+\tau v) \diff(V[\mu] - V_{\mu}^{\varepsilon})(x,v) \,  \\ &+
\int_{\R^d \times \R^d} \psi(x+\tau v) \diff(V_{\nu}^{\varepsilon}-V[\nu])(x,v) +\int_{\R^d \times \R^d} \psi(x+\tau v) \diff(V_{\mu}^{\varepsilon}-V_{\nu}^{\varepsilon})(x,v). 
\end{split}
\end{equation}
First, we note that the change of variable formula for measures (see e.g. \cite[Theorem 3.6.1]{Bogachev1.2007}) implies
    \begin{align*}
       &\int_{\R^d \times \R^d} \psi(x+\tau v) \diff(V_{\mu}^{\varepsilon}-V_{\nu}^{\varepsilon})(x,v)\\
       =&\int_{\R^d \times \R^d} \psi(x+\tau v) \diff V_{\mu}^{\varepsilon}(x,v)-\int_{\R^d \times \R^d} \psi(x+\tau v) \diff V_{\nu}^{\varepsilon}(x,v)\\
       =& \int_{\R^d \times \R^d} \psi(x+\tau v) \diff \!\left(\pi^{\#}_{1,2}\,p^{\varepsilon}\right)(x,v)-\int_{\R^d \times \R^d} \psi(y+\tau w) \diff\!\left(\pi^{\#}_{3,4}\, p^{\varepsilon}\right)(y,w)\\
       =&\int_{\left(\R^d\right)^4} \psi(x+\tau\,v)\diff p^{\varepsilon}(x,v,y,w)-\int_{\left(\R^d\right)^4} \psi(y+\tau\,w)\diff p^{\varepsilon}(x,v,y,w)\\
       =&\int_{\left(\R^d\right)^4} \left(\psi(x+\tau\,v) - \psi(y+\tau\,w)\right) \diff p^{\varepsilon}(x,v,y,w).
    \end{align*} 
\noindent Here, $\pi_{1,2}$ and $\pi_{3,4}$ denote the projections to the first  and last two coordinates, respectively. Notice carefully that we introduced the second pair of variables $(y,w)$.  To handle the first term appearing in \eqref{eq:split_proof_pic_implies_ours} we  first note that
$V[\mu]-V_{\mu}^{\varepsilon}$ and $\mu-\mu^{\varepsilon}$ are nonnegative measures by construction so that we can apply Lemma \ref{lem:connection_mu_Vmu} and the linearity of the total variation norm to see
\begin{align*}
    \|V[\mu]-V^{\epsilon}_{\mu}\|_{TV}=\|V[\mu]\|_{TV}-\|V^{\epsilon}_{\mu}\|_{TV}
    =\|\mu\|_{TV}-\|\mu^{\varepsilon}\|_{TV}=\|\mu-\mu^{\varepsilon}\|_{TV}.
\end{align*}
Thus, we have for the first term
$$
\left|\int_{\R^d \times \R^d} \psi(x+\tau v) \diff(V[\mu] - V_{\mu}^{\varepsilon})(x,v)\right| \leq \| \mu - \mu^{\varepsilon}\|_{TV}
$$
because $\|\psi\|_{\infty} \leq 1$. Similarly,
$$
\left|\int_{\R^d \times \R^d} \psi(y+\tau w) \diff(V_{\nu}^{\varepsilon}-V[\nu])(y,w)\right|=\left|\int_{\R^d \times \R^d} \psi(y+\tau w) \diff(V[\nu] - V_{\nu}^{\varepsilon})(y,w)\right| \leq \| \nu - \nu^{\varepsilon}\|_{TV}.
$$
For the last term in \eqref{eq:split_proof_pic_implies_ours} we have
$$
\left|\psi(x+\tau\,v) - \psi(y+\tau\,w)\right| \leq |x-y| + \tau\, |v-w|.
$$
Therefore, using the definition of the transference plan we obtain
\begin{multline*}
\left|\int_{\left(\R^d\right)^4} \left(\psi(x+\tau\,v) - \psi(y+\tau\,w)\right) \diff p^{\varepsilon}(x,v,y,w)\right|
\leq 
\int_{\left(\R^d\right)^4} 
\left(|x-y| + \tau\, |v-w|\right)
\diff p^{\varepsilon}(x,v,y,w) \\
\phantom{\int_{\left(\R^d\right)^4}}
\leq W_1(\mu^{\varepsilon}, \nu^{\varepsilon}) + \tau \,\mathcal{W}^g(V[\mu],V[\nu]) + \tau \, \varepsilon \leq W_1(\mu^{\varepsilon}, \nu^{\varepsilon}) + \tau \,C_F\|\mu-\nu\|_{BL^*} + \tau \, \varepsilon,
\end{multline*}
where we applied condition \eqref{ass:Lipschitz cont V wrt Wg} in the last step. It follows that
\begin{equation*}
\begin{split}
 \bigg|\int_{\R^d \times \R^d} \psi(x+\tau& v)  \diff(V[\mu] - V[\nu])(x,v)\bigg|  \\ 
&  \leq  
\| \mu - \mu^{\varepsilon}\|_{TV} + 
\| \nu - \nu^{\varepsilon}\|_{TV} + 
 W_1(\mu^{\varepsilon}, \nu^{\varepsilon}) + \tau \,C_F\|\mu-\nu\|_{BL^*} + \tau \, \varepsilon \phantom{\int_{\R^d}} \\
 &=  \|\mu - \nu\|_{BL^*} +  \tau \,C_F\|\mu-\nu\|_{BL^*} + \tau \, \varepsilon   = (1 + \tau \,C_F)\, \|\mu-\nu\|_{BL^*} + \tau \, \varepsilon. \phantom{\int_{\R^d}}
\end{split}
\end{equation*}
As $\varepsilon > 0$ can be arbitrarily small, the conclusion follows.
\end{proof}

\begin{remark}
In the case of conservative problem in the space of probability measure as in \cite{MR3961299} the setting above can be substantially simplified. First, the definition of $\mathcal{W}^g$ in \eqref{def:curly W g} boils down to
    \begin{align}
        \label{def:curlyWg_probability}
        \mathcal{W}^g(V_1,V_2):=\inf\bigg\{\int_{(\R^d)^4}|v-w|\diff p(x,v,y,w) \mid  p\in \mathcal{P}(V_1, V_2), \,  \pi_{13}^{\#}\,p\in \mathcal{P}^{\text{opt}}(\tilde\mu_1,\tilde \mu_2)\bigg\}.
    \end{align}
because $V_1$, $V_2$ are probability measures cf. \cite[Definition 4.1]{MR3961299}. Moreover, continuity conditions simplify to 
    \begin{align}
    \label{ass:Lipschitz cont V wrt Wg_cons}
    \tag{\textbf{$\widehat{V_{2,3,c}}$}}
         \mathcal{W}^g(V[\mu],V[\nu])\leq C_F \, W_1(\mu, \nu),
    \end{align}
    \begin{align}
    \label{ass:LipschitzcontV_cons}
    \tag{$V_{2,c}$}
    W_1(V[\mu], V[\nu]) \leq C_F\, W_1(\mu, \nu),
    \end{align}
\begin{align}\label{eq:new_ass_BL_cons}
\tag{$V_{3,c}$}
\sup_{|\psi|_{\Lip} \leq 1}\int_{\R^d \times \R^d} \psi(x + \tau \, v) \diff(V[\mu] - V[\nu])(x,v) \leq (1 + C_H(R)\,\tau) \, W_1(\mu, \nu).
\end{align}
for \eqref{ass:Lipschitz cont V wrt Wg}, \eqref{ass:Lipschitz cont V}, \eqref{eq:new_ass_BL} respectively. Then the same proof as in Lemma \ref{lem:piccoli_implies_ours} shows that \eqref{ass:Lipschitz cont V wrt Wg_cons} implies \eqref{ass:LipschitzcontV_cons} and \eqref{eq:new_ass_BL_cons}.
\end{remark}

\newpage
\bibliographystyle{abbrv}
\bibliography{fastlimit}

\begin{thebibliography}{10}

\bibitem{MR4045015}
A.~S. Ackleh, N.~Saintier, and J.~Skrzeczkowski.
\newblock Sensitivity equations for measure-valued solutions to transport
  equations.
\newblock {\em Math. Biosci. Eng.}, 17(1):514--537, 2020.

\bibitem{MR3714980}
A.~Aydo\u{g}du, S.~T. McQuade, and N.~Pouradier~Duteil.
\newblock Opinion dynamics on a general compact {R}iemannian manifold.
\newblock {\em Netw. Heterog. Media}, 12(3):489--523, 2017.

\bibitem{Bogachev1.2007}
V.~I. Bogachev.
\newblock {\em Measure theory. {V}ol. {I}}.
\newblock Springer-Verlag, Berlin, 2007.

\bibitem{MR3138105}
{\AA}.~Br{\"a}nnstr{\"o}m, L.~Carlsson, and D.~Simpson.
\newblock On the convergence of the escalator boxcar train.
\newblock {\em SIAM J. Numer. Anal.}, 51(6):3213--3231, 2013.

\bibitem{MR4028475}
S.~Cacace, F.~Camilli, R.~De~Maio, and A.~Tosin.
\newblock A measure theoretic approach to traffic flow optimisation on
  networks.
\newblock {\em European J. Appl. Math.}, 30(6):1187--1209, 2019.

\bibitem{MR4206990}
F.~Camilli, G.~Cavagnari, R.~De~Maio, and B.~Piccoli.
\newblock Superposition principle and schemes for measure differential
  equations.
\newblock {\em Kinet. Relat. Models}, 14(1):89--113, 2021.

\bibitem{MR3657111}
F.~Camilli, R.~De~Maio, and A.~Tosin.
\newblock Transport of measures on networks.
\newblock {\em Netw. Heterog. Media}, 12(2):191--215, 2017.

\bibitem{MR3779635}
F.~Camilli, R.~De~Maio, and A.~Tosin.
\newblock Measure-valued solutions to nonlocal transport equations on networks.
\newblock {\em J. Differential Equations}, 264(12):7213--7241, 2018.

\bibitem{carrillo2012}
J.~A. Carrillo, R.~M. Colombo, P.~Gwiazda, and A.~Ulikowska.
\newblock Structured populations, cell growth and measure valued balance laws.
\newblock {\em J. Differential Equations}, 252(4):3245--3277, 2012.

\bibitem{MR3986559}
J.~A. Carrillo, P.~Gwiazda, K.~Kropielnicka, and A.~K. Marciniak-Czochra.
\newblock The escalator boxcar train method for a system of age-structured
  equations in the space of measures.
\newblock {\em SIAM J. Numer. Anal.}, 57(4):1842--1874, 2019.

\bibitem{carrillo2014}
J.~A. Carrillo, P.~Gwiazda, and A.~Ulikowska.
\newblock Splitting-particle methods for structured population models:
  convergence and applications.
\newblock {\em Math. Models Methods Appl. Sci.}, 24(11):2171--2197, 2014.

\bibitem{MR1207136}
G.~Da~Prato and J.~Zabczyk.
\newblock {\em Stochastic equations in infinite dimensions}, volume~44 of {\em
  Encyclopedia of Mathematics and its Applications}.
\newblock Cambridge University Press, Cambridge, 1992.

\bibitem{deRoos1988}
A.~M. de~Roos.
\newblock Numerical methods for structured population models: the escalator
  boxcar train.
\newblock {\em Numer. Methods Partial Differential Equations}, 4(3):173--195,
  1988.

\bibitem{our_book_ACPJ}
C.~D\"{u}ll, P.~Gwiazda, A.~Marciniak-Czochra, and J.~Skrzeczkowski.
\newblock {\em Spaces of measures and their applications to structured
  population models}, volume~36 of {\em Cambridge Monographs on Applied and
  Computational Mathematics}.
\newblock Cambridge University Press, Cambridge, 2022.

\bibitem{MR3409135}
L.~C. Evans and R.~F. Gariepy.
\newblock {\em Measure theory and fine properties of functions}.
\newblock Textbooks in Mathematics. CRC Press, Boca Raton, FL, revised edition,
  2015.

\bibitem{MR3342408}
J.~H.~M. Evers, S.~C. Hille, and A.~Muntean.
\newblock Mild solutions to a measure-valued mass evolution problem with flux
  boundary conditions.
\newblock {\em J. Differential Equations}, 259(3):1068--1097, 2015.

\bibitem{MR3507552}
J.~H.~M. Evers, S.~C. Hille, and A.~Muntean.
\newblock Measure-valued mass evolution problems with flux boundary conditions
  and solution-dependent velocities.
\newblock {\em SIAM J. Math. Anal.}, 48(3):1929--1953, 2016.

\bibitem{Folland.1984}
G.~B. Folland.
\newblock {\em Real analysis}.
\newblock Pure and Applied Mathematics (New York). John Wiley \& Sons, Inc.,
  New York, 1984.
\newblock Modern techniques and their applications, A Wiley-Interscience
  Publication.

\bibitem{MR4027078}
P.~Gwiazda, S.~C. Hille, K.~{\L}yczek, and A.~{\'{S}}wierczewska-Gwiazda.
\newblock Differentiability in perturbation parameter of measure solutions to
  perturbed transport equation.
\newblock {\em Kinet. Relat. Models}, 12(5):1093--1108, 2019.

\bibitem{GwJaMaUl2013}
P.~Gwiazda, J.~Jab{\l}o{{\'n}}ski, A.~Marciniak-Czochra, and A.~Ulikowska.
\newblock Analysis of particle methods for structured population models with
  nonlocal boundary term in the framework of bounded {L}ipschitz distance.
\newblock {\em Numer. Methods Partial Differential Equations},
  30(6):1797--1820, 2014.

\bibitem{MR3461738}
P.~Gwiazda, K.~Kropielnicka, and A.~Marciniak-Czochra.
\newblock The escalator boxcar train method for a system of age-structured
  equations.
\newblock {\em Netw. Heterog. Media}, 11(1):123--143, 2016.

\bibitem{MR2644146}
P.~Gwiazda, T.~Lorenz, and A.~Marciniak-Czochra.
\newblock A nonlinear structured population model: {L}ipschitz continuity of
  measure-valued solutions with respect to model ingredients.
\newblock {\em J. Differential Equations}, 248(11):2703--2735, 2010.

\bibitem{MR2746205}
P.~Gwiazda and A.~Marciniak-Czochra.
\newblock Structured population equations in metric spaces.
\newblock {\em J. Hyperbolic Differ. Equ.}, 7(4):733--773, 2010.

\bibitem{gwiazda2021}
P.~Gwiazda, B.~Miasojedow, J.~Skrzeczkowski, and Z.~Szyma\'nska.
\newblock Convergence of the {EBT} method for a non-local model of cell
  proliferation with discontinuous interaction kernel.
\newblock {\em arXiv preprint arXiv:2106.05115}, 2021.

\bibitem{MeDi1986}
J.~A.~J. Metz and O.~Diekmann.
\newblock Age dependence.
\newblock In {\em The dynamics of physiologically structured populations
  ({A}msterdam, 1983)}, volume~68 of {\em Lecture Notes in Biomath.}, pages
  136--184. Springer, Berlin, 1986.

\bibitem{MR3961299}
B.~Piccoli.
\newblock Measure differential equations.
\newblock {\em Arch. Ration. Mech. Anal.}, 233(3):1289--1317, 2019.

\bibitem{piccoli2014generalized}
B.~Piccoli and F.~Rossi.
\newblock Generalized {W}asserstein distance and its application to transport
  equations with source.
\newblock {\em Arch. Ration. Mech. Anal.}, 211(1):335--358, 2014.

\bibitem{MR4026977}
B.~Piccoli and F.~Rossi.
\newblock Measure dynamics with probability vector fields and sources.
\newblock {\em Discrete Contin. Dyn. Syst.}, 39(11):6207--6230, 2019.

\bibitem{rossi2016control}
F.~Rossi, N.~P. Duteil, N.~Yakoby, and B.~Piccoli.
\newblock Control of reaction-diffusion equations on time-evolving manifolds.
\newblock In {\em 2016 IEEE 55th Conference on Decision and Control (CDC)},
  pages 1614--1619. IEEE, 2016.

\bibitem{MR4066016}
J.~Skrzeczkowski.
\newblock Measure solutions to perturbed structured population
  models---differentiability with respect to perturbation parameter.
\newblock {\em J. Differential Equations}, 268(8):4119--4182, 2020.

\bibitem{MR3244279}
H.~H. Sohrab.
\newblock {\em Basic real analysis}.
\newblock Birkh\"{a}user/Springer, New York, second edition, 2014.

\bibitem{szymanska2021}
Z.~Szyma\'nska, B.~Miasojedow, J.~Skrzeczkowski, and P.Gwiazda.
\newblock Bayesian inference of a non-local proliferation model.
\newblock {\em arXiv preprint arXiv:2106.05955}, pages 1--29, 2021.

\bibitem{Th2003}
H.~R. Thieme.
\newblock {\em Mathematics in population biology}.
\newblock Princeton Series in Theoretical and Computational Biology. Princeton
  University Press, Princeton, NJ, 2003.

\bibitem{MR2997595}
A.~Ulikowska.
\newblock An age-structured two-sex model in the space of {R}adon measures:
  well posedness.
\newblock {\em Kinet. Relat. Models}, 5(4):873--900, 2012.

\bibitem{villani:2003}
C.~Villani.
\newblock {\em Topics in optimal transportation}, volume~58 of {\em Graduate
  Studies in Mathematics}.
\newblock American Mathematical Society, Providence, RI, 2003.

\bibitem{Webb}
G.~Webb.
\newblock {\em Theory of nonlinear age-dependent population dynamics}.
\newblock Marcel Dekker, Inc., 1985.

\end{thebibliography}

\end{document}